\documentclass[12pt,psamsfonts]{amsart}
\usepackage{amssymb,eucal,epsfig}
\usepackage{verbatim,enumerate,mathdots}    
\usepackage[all,graph,frame]{xy}
\usepackage{amsthm}



\def\ad{\operatorname{ad}}
\def\pip{\operatorname{\pi_\perp}}

\def\rk{\operatorname{rk}}          
\def\Span{\operatorname{Span}}      
\def\End{\operatorname{End}}        

\def\b{\mathfrak b}

\def\g{\mathfrak g}
\def\h{\mathfrak h}
\def\hp{{\mathfrak h}^{\perp}}
\def\k{\mathfrak k}
\def\m{\mathfrak m}
\def\n{\mathfrak n}

\def\V{\mathcal V}
\def\N{\mathbb N}

\def\R{\mathbb R}
\def\C{\mathbb C}
\def\0{\mathbb 0}
\def\Kf{\mathbb K}
\def\E{\mathcal E}
\def\B{\mathcal B}
\def\K{\delta}                      
\def\<{\langle}                     
\def\>{\rangle}
\def\ip{\<\cdot,\cdot\>}
\def\al{\alpha}

\theoremstyle{plain}
\newtheorem{theorem}{Theorem}[section]

\newtheorem{lemma}[theorem]{Lemma}

\newtheorem*{theorem*}{Theorem}

\theoremstyle{definition}
\newtheorem{Definition}[theorem]{Definition}
\newtheorem{remark}[theorem]{Remark}

\title{Totally Geodesic Subalgebras of Filiform Nilpotent Lie algebras}
\author[Cairns, Hini\'c Gali\'c and  Nikolayevsky]{Grant Cairns $^1$, Ana Hini\'c Gali\'c $^2$  and Yuri Nikolayevsky $^3$}

\address{$^1$
Department of Mathematics and Statistics\\
 La Trobe University\\
 Melbourne,  3086\\
 Australia}
\email{G.Cairns@latrobe.edu.au}

\address{$^2$
Department of Mathematics and Statistics\\
 La Trobe University\\
 Melbourne,  3086\\
 Australia}
\email{A.HinicGalic@latrobe.edu.au}

\address{$^3$
Department of Mathematics and Statistics\\
 La Trobe University\\
Melbourne,  3086\\
Australia}
\email{Y.Nikolayevsky@latrobe.edu.au}

\begin{document}

\maketitle

\begin{abstract}
We determine the maximal dimension of totally geodesic subalgebras of $\N$-graded filiform Lie algebras, and we show that these bounds are attained.
\end{abstract}

\section{Introduction}

Consider a finite-dimensional Lie algebra $\g$. If $\g$ is equipped with an inner product $\ip$ and  $\h$ is a Lie subalgebra of $\g$, we denote its orthogonal complement $\hp$. We say that $\h$ is {\em totally geodesic} if
\begin{equation}\label{L:basic}
2 \langle \nabla_Y Z,X\rangle :=\langle [X,Y],Z\rangle +\langle [X,Z],Y\rangle = 0, \ \text{for all}\ X\in\hp, Y,Z\in \h.
\end{equation}
 The definition is chosen so that the corresponding Lie subgroup is totally geodesic in the usual sense.

This paper is the sequel to \cite{CHGN}, in which we gave a number of results concerning totally geodesic subalgebras of  nilpotent Lie algebras. In particular, we showed that in each dimension $n\geq 3$, there is up to isomorphism only one  filiform nilpotent Lie algebra that possesses a  totally geodesic subalgebra of codimension two.
We showed that in filiform nilpotent Lie algebras, totally geodesic subalgebras that leave invariant their orthogonal complements have dimension  at most half the dimension of the algebra. And we gave an example of a 6-dimensional filiform nilpotent Lie algebra that has no totally geodesic subalgebra of dimension $>2$, for any choice of inner product. Results on bases of geodesics in Lie algebras were given in \cite{CLNN}.

In this paper, we focus on an important natural family of nilpotent Lie algebras.

\begin{Definition}
We say  an $n$-dimensional nilpotent Lie algebra $\g$ is {\em $\N$-graded filiform}, if it can be decomposed in a direct
sum of one-dimensional subspaces
$\g=\bigoplus_{i=1}^n V_i$ with  $[V_1,V_i]=V_{i+1}$ for all $i>1$ and $[V_i,V_j]\subset V_{i+j}$ for all $i,j\in\N$, where for convenience we set $V_i=0$ for $i>n$.
\end{Definition}

These algebras have been completely classified by Millionshchikov  \cite{DMil2004}. There are 6 natural sequences of algebras of arbitrary large dimension, and in addition, in each dimension from 7 through to 11, there is a one-parameter family of exceptional algebras. The purpose of our paper is to establish the following result:

\begin{theorem}
Suppose that $\h$ is a subalgebra of an $\N$-graded filiform Lie algebra $\g$ and that $\h$ is totally geodesic with respect to some inner product on~$\g$. Then one of the following conditions holds:
\begin{enumerate}[\rm (a)]
\item $\g$ has a presentation with basis $\{X_1,\dots,X_n\}$ and relations $[X_1,X_i]=X_{i+1} $ for all $1<i<n$. In this case $\dim(\h)\leq \dim(\g)- 2$.
\item $\g$ has a presentation with basis $\{X_1,\dots,X_{2k+1}\}$  and relations $[X_1,X_i]=X_{i+1} $ for all $1<i<2k+1$ and  $[X_l, X_{2k+1-l}] = (-1)^{l+1}X_{2k+1}$ for all $1<l <k+1$.  In this case    $\dim(\h)\leq \dim(\g)-4$.
\item $\g$ is not isomorphic to one of the algebras in parts {\rm (a)} and {\rm (b)}, in which case\break $\dim(\h)\leq \lfloor \dim (\g) /2\rfloor$.
\end{enumerate}
Moreover, each of the above bounds is attained for some inner product and  subalgebra $\h$.
 \end{theorem}

The paper is organised as follows. In the next section we recall some background information, including Millionshchikov's classification. In Section \ref{S:prel} we establish some general preliminary results. Section \ref{S:table1}
 deals with the 6 natural families of $\N$-graded filiform Lie algebras. Finally, Section \ref{S:table2}
 treats the exceptional algebras in dimension~7 through to~11.


\section{Background}

In 1983, Fialowski classified all infinite-dimensional $\N$-graded filiform Lie algebras, and later some of her results were rediscovered by Khakimdjanova and Khakimdjanov in \cite{YKKK}.

\begin{theorem}[\textbf{\cite{AF}}]\label{T:AF}
Let $\g $ be an infinite-dimensional $\N$-graded filiform Lie algebra. Then $\g$ is isomorphic to precisely one of the following  Lie algebras:

\begin{enumerate}[\rm (a)]
\item $\m_0=\Span( X_i, i\in\N \mid [X_1,X_j]=X_{j+1}, \forall j\geq 2)$,
\item $\m_2=\Span( X_i, i\in\N \mid [X_1,X_j]=X_{j+1}, \forall j\geq 2,\ [X_2,X_j]=X_{j+2}, \forall j\geq
3)$,
\item $\V=\Span( X_i, i\in\N \mid [X_i,X_j]=(j-i)X_{i+j}, \forall i , j)$.
\end{enumerate}
\end{theorem}

If $\g=\bigoplus_{i=1}^\infty V_i$ is an infinite-dimensional $\N$-graded filiform Lie algebra, then the quotient algebra $\g(n)= \g/\bigoplus_{i={n+1}}^\infty V_i$ is an $n$-dimensional $\N$-graded filiform Lie algebra. Therefore, in this way we can obtain three natural sequences of finite-dimensional algebras, that we denote: $\m_0(n)$, $\m_2(n)$, $\V_n$.
In 1991, Khakimdjanov proved in \cite{YK} that there only exists  a finite number of non-isomorphic $\N$-graded filiform Lie algebras over $\C$ in dimensions $\geq 12$. Then,  in 2004, Millionshchikov classified all finite-dimensional $\N$-graded filiform Lie algebras  over an arbitrary field $\Kf$ of characteristic zero \cite{DMil2004}. 

\begin{theorem}[\textbf{\cite{DMil2004}}]\label{T:DMil}
Let $\g $ be a finite-dimensional $\N$-graded  filiform Lie algebra. Then $\g$ is isomorphic to a Lie algebra from the following list:
\begin{enumerate}[\rm (a)]
\item the six sequences $\m_0(n)$, $\m_2(n)$, $\V_n$, $\m_{0,1}(2k+1)$, $\m_{0,2}(2k+2)$, $\m_{0,3}(2k+3)$, defined by the basis $\{X_1,\ldots,X_n\}$ and commutation relations given in Table~\ref{Table:1};

\item the 5 one-parameter families $\g_{n,\alpha}$ of dimensions $n = 7,\ldots, 11$ respectively,
defined by their basis and commutation relations given in Table~\ref{Table:2}.

\end{enumerate}
\end{theorem}

{\begin{table}
\renewcommand{\arraystretch}{1.4}
 \begin{center}
  \begin{tabular}{c|c|l}\hline
   {\bf algebra} &{\bf dimension}&{\bf presentation} \\ \hline
   $\m_0(n)$ & $n\geq 3$ & $[X_1,X_i]=X_{i+1}$,\hfill $i=2,\ldots,n-1$ \\      \hline
   $\m_2(n)$ & $n\geq 5$& $[X_1,X_i]=X_{i+1},\hfill i=2,\ldots,n-1$ \\
   &  & $[X_2,X_i]=X_{i+2},\hfill i=3,\ldots,n-2$ \\      \hline

   $\V_n$ & $n\geq 12$& $[X_i,X_j]=\left\{
                                     \begin{array}{ll}
                                       (j-i)X_{i+j}, & \hbox{$i+j\leq n$;} \\
                                       0, & \hbox{$i+j> n$;}
                                     \end{array}
                                   \right. $\\ \hline
   $\m_{0,1}(2k+1)$  & $n=2k+1 $,  & $[X_1,X_i]=X_{i+1}$,\hfill $i=2,\ldots,2k$ \\
    &$k\geq 3$ &                     $[X_l, X_{2k-l+1}] = (-1)^{l+1}X_{2k+1}$,\hfill $l = 2,\ldots,k$         \\                                               \hline
 $\m_{0,2}(2k+2)$  & $n=2k+2 $,  & $[X_1,X_i]=X_{i+1}$,\hfill $i=2,\ldots,2k+1$ \\
    &$k\geq 3$ &                  $[X_l, X_{2k-l+1}] = (-1)^{l+1}X_{2k+1}$,\hfill $l = 2,\ldots,k$         \\
  & &                             $[X_j, X_{2k-j+2}] = (-1)^{j+1}(k-j+1)X_{2k+2}$,\hfill $j = 2,\ldots,k$ \\         \hline
    $\m_{0,3}(2k+3)$  & $n=2k+3, $  & $[X_1,X_i]=X_{i+1}$,\hfill $i=2,\ldots,2k+2$ \\
    & $k\geq 3$&                  $[X_l, X_{2k-l+1}] = (-1)^{l+1}X_{2k+1}$,\hfill $l = 2,\ldots,k$         \\
  & &                             $[X_j, X_{2k-j+2}] = (-1)^{j+1}(k-j+1)X_{2k+2}$,\hfill $j = 2,\ldots,k$ \\
 & &                              $[X_m, X_{2k-m+3}] = (-1)^{m}((m-2)k-\frac{(m-2)(m-1)}2)X_{2k+3}$, \\
& & \hfill $m = 3,\ldots,k+1$ \\
 \hline
 \end{tabular}
 \end{center}
 \vspace{1mm} \caption{Six infinite sequences of $\N$-graded  filiform Lie algebra}\label{Table:1}
\end{table}}

{\begin{table}[h]
\renewcommand{\arraystretch}{1.4}
 \begin{center}
  \begin{tabular}{c|c|l}\hline
   {\bf algebra} &{\bf restrictions}&{\bf presentation} \\ \hline
   $\g_{7,\alpha}$ & $\alpha\neq -2$ & $[X_1,X_j]=X_{j+1},\hfill 2\leq j\leq 6$ \\
  &  & $[X_2,X_3]=(2+\alpha)X_5$,\; $[X_2,X_4]=(2+\alpha)X_6$,  \\
  &  & $[X_2,X_5]=(1+\alpha)X_7$, \; $[X_3,X_4]=X_7$, \\    \hline
   $\g_{8,\alpha}$ & $\alpha\neq -2$ &  relations of $\g_{7,\alpha}$ and: \\
  &  & $[X_1,X_7]=X_8$,\; $[X_2,X_6]=\alpha X_8,$ \; $[X_3,X_5]=X_8$, \\  \hline
   $\g_{9,\alpha}$ & $\alpha\not=-\frac 52, -2$ & relations of $\g_{8,\alpha}$ and: \\
  &  & $[X_1,X_8]=X_9$,\quad $[X_2, X_7] = \frac{2\alpha^2+3\alpha-2}{2\alpha+5}X_9$,\\
& &  $[X_3, X_6] = \frac{2\alpha+2}{2\alpha+5}X_9$,\quad $[X_4, X_5] = \frac{3}{2\alpha+5}X_9$, \\  \hline
   $\g_{10,\alpha}$ & $\alpha\not=-\frac 52$ & relations of $\g_{9,\alpha}$  and: \\
  &  &  $[X_1,X_9]=X_{10}$,\quad $[X_2, X_8] = \frac{2\alpha^2+\alpha-1}{2\alpha+5}X_{10}$,\\
& &    $[X_3, X_7] = \frac{2\alpha-1}{2\alpha+5} X_{10}$, \quad $[X_4, X_6] = \frac{3}{2\alpha+5}X_{10}$, \\  \hline
   $\g_{11,\alpha}$ & $\alpha\not=-\frac 52,-1,-3$ & relations of $\g_{10,\alpha}$  and: \\
  & &       $[X_1,X_{10}]=X_{11}$,\quad $[X_2, X_9] = \frac{2\alpha^3+2\alpha^2+3}{2(\alpha^2+4\alpha+3)} X_{11}$,\\
 & & $[X_3, X_8] =\frac{4\alpha^3+8\alpha^2-8\alpha-21}{2(\alpha^2+4\alpha+3)(2\alpha+5)}X_{11}$, \\
 & & $[X_4, X_7] = \frac{ 3(2\alpha^2+4\alpha+5)}{2(\alpha^2+4\alpha+3)(2\alpha+5)}X_{11}$,\\
 & & $[X_5, X_6] = \frac{3(4\alpha+1)}{2(\alpha^2+4\alpha+3)(2\alpha+5)}X_{11}$ \\
      \hline
  \end{tabular}
 \end{center}
 \vspace{1mm} \caption{Five one-parameter families of $\N$-graded filiform Lie algebra}\label{Table:2}
\end{table}}

\begin{remark}\label{R:DMilTheorem}
 For $\alpha=-2$ we have $\g_{7,-2}\cong \m_{0,1}(7)$, $\g_{8,-2}\cong \m_{0,2}(8)$ and $\g_{9,-2}\cong \m_{0,3}(9)$. These isomorphisms seem to have been overlooked in \cite{DMil2004}. Apart from these, the algebras in Tables~\ref{Table:1} and \ref{Table:2} are pair-wise non-isomorphic. However, if we drop the restrictions on dimensions of the algebras in Table~\ref{Table:1}, then we also have the following isomorphisms: $\m_0(3)\cong\m_2(3)\cong\V_3$, $\m_0(4)\cong\m_2(4)\cong\V_4$, $\m_2(5)\cong\V_5$, $\m_2(6)\cong\V_6$, and for $\alpha=8$ we have $\g_{n,8}\cong \V_n$ where $n=7,\ldots,11$, as one can see by examining the basis $\{X_1, \frac{1}{(k-2)!\cdot 60}X_k :k=2,\ldots,n\}$.
\end{remark}


\section{Preliminaries}\label{S:prel}

Throughout this paper, $\g$ is an $\N$-graded filiform Lie algebra of dimension $n$. We fix the decomposition $\g=\bigoplus_{i=1}^n V_i$ with  $[V_1,V_i]=V_{i+1}$ for all $i>1$ and $[V_i,V_j]\subset V_{i+j}$ for all $i,j\in\N$, where for convenience, we set $V_i=0$ for $i>n$. Introduce the ideals $\g_i:=\oplus_{j\geq i} V_j$, for $i=1, \dots, n$. For $Y\in\g$, we define the \emph{degree} of $Y$, denoted $\deg(Y)$, to be the largest natural number $k$ such that $Y\in\g_k$, and  for convenience, we set $\deg(0)=\infty$.

We choose a basis $\B=\{X_1,\dots,X_n\}$ for $\g$ with $X_i\in V_i$ for  $i=1,\dots,n$. Obviously, $\deg([X_i,X_j])\geq i+j$, for all $i,j$. Consequently, if $Y_1,Y_2\in\g$ such that $\deg(Y_1)=i$ and $\deg(Y_2)=j$, then $\deg ([Y_1,Y_2])\geq i+j$.
Moreover, if $\deg([X_i,X_j])=i+j$, then $\deg([Y_1,Y_2])=i+j$.

\begin{remark}
If we take an inner product on $\g$ for which $X_1,\dots,X_n$ are orthonormal, then the subalgebra $\h$ generated by $\{X_i: i \text{ is even}\}$ is a totally geodesic subalgebra of dimension $\left\lfloor n/2\right\rfloor$; see \cite{KP}.
\end{remark}

Now assume $\g$ is equipped with an arbitrary inner product $\ip$. Applying the Gram-Schmidt orthonormalisation procedure to $\B$, starting with the element of the largest  degree, we obtain an orthonormal basis $\E=\{E_1,\dots,E_n\}$, where for each $i$ one has $\deg (E_i)=i$. Clearly, $\Span(E_k,\ldots,E_n) = \g_k $. Note that by construction, $[E_i,E_j]\in\g_{i+j}$ for all $i,j$. Furthermore, $\langle[E_1,E_i],E_{i+1}\rangle\neq 0$ for all $1<i<n$. So $E_1$ has maximal nilpotency. For more details see \cite{CHGN}.

Now let $\h$ be a totally geodesic subalgebra of $\g$ of dimension greater than~$1$. We will repeatedly use the following facts:

\begin{lemma}\label{L:facts} We have:
\begin{enumerate}[\rm (a)]
\item\label{I:e1} $E_1\in \hp$ and $\h\subset \Span(X_2,\ldots,X_n)$.
\item\label{I:eiei+1}
It is impossible that $E_i,E_{i+1}\in\h$, for any $i=2,\dots,n-1$. In particular, there is an element of degree 2 or 3 in $\hp$.
\item \label{I:ei+aei+1}
If $E_i+aE_{i+1}\in\h$, for any $i=2,\dots,n-1$, then $a=0$.
\item\label{I:kn-k}
Suppose $E_n,Y\in\h$ where $\deg(Y)=k$ for some $k=2,\ldots, n-1$. If $[X_k,X_{n-k}]=aX_n$ $(a\neq0)$, then there is no $Z\in\hp$ such that $\deg(Z)=n-k$.
\item\label{I:kn-k-1}
Suppose $E_{n-1},Y\in\h$ where $\deg(Y)=k$, for some $k=2,\ldots, n-3$. If $[X_k,X_{n-k-1}]=bX_{n-1}$ $(b\neq0)$, then there is no $Z\in\hp$ such that $\deg(Z)=n-k-1$.
\end{enumerate}
\end{lemma}

\begin{proof}
\eqref{I:e1} By \cite[Lemma 4.4]{CHGN}, the elements of degree one have maximal nilpotency. Then by \cite[Lemma 4.6]{CHGN}, $\h\subset \g_2$. Hence $E_1\in \hp$; see  \cite[Remark 4.7]{CHGN}. So $\h\subset  \Span(E_2,\ldots,E_n)= \Span(X_2,\ldots,X_n)$.

\eqref{I:eiei+1} Suppose that $E_i,E_{i+1}\in\h$, for some $i=2,\dots,n-1$. As $E_1\in\hp$ then
\[
2\langle\nabla_{E_i} E_{i+1},E_1\rangle=\langle[E_1,E_i],E_{i+1}\rangle+\langle[E_1,E_{i+1}],E_i\rangle=\langle[E_1,E_i],E_{i+1}\rangle\neq 0,
\]
which would contradict \eqref{L:basic}. Therefore, there exists an element $Z\in\hp$ of degree 2 or 3, since otherwise  $E_2,E_3\in\h$.

\eqref{I:ei+aei+1} From \eqref{L:basic} we have $0=\langle[E_1,E_i+aE_{i+1}],E_i+aE_{i+1}\rangle=a\langle[E_1,E_i],E_{i+1}\rangle$, so $a=0$.

\eqref{I:kn-k} By \eqref{L:basic}, if $Y,E_n\in\h$, $\deg(Y)=k$ and there is $Z\in\hp$ such that $\deg(Z)=n-k$, then we have
$$
0=\langle[Z,Y],E_{n}\rangle+\langle[Z,E_{n}],Y\rangle=\langle[Z,Y],E_{n}\rangle.
$$ However, since $\deg([X_k,X_{n-k}])=n$, we have $\deg([Z,Y])=n$, which is a  contradiction.

\eqref{I:kn-k-1}   By \eqref{L:basic}, if $Y,E_{n-1}\in\h$, $\deg(Y)=k$ and there is $Z\in\hp$ such that $\deg(Z)=n-k-1$, then we have
$$
0=\langle[Z,Y],E_{n-1}\rangle+\langle[Z,E_{n-1}],Y\rangle=\langle[Z,Y],E_{n-1}\rangle.
$$ But since $\deg([X_k,X_{n-k-1}])=n-1$, we have  $\deg([Z,Y])=n-1$, which is a  contradiction.
\end{proof}


\begin{lemma}\label{L:zn-2,zn-1}
Suppose that $X_n\not\in\h$ and  there are elements of degree $p,\ldots,n-1$ in $\h$, for some $p\geq 2$. Then there are no elements of degree $p+2,\ldots,n$ in $\hp$.
\end{lemma}

\begin{proof}
If $X_n\not\in\h$ and there are elements $Y_i\in\h$ with $\deg(Y_i)=i$ for $i=p,\ldots,n-1$, then by taking linear combinations if necessary, we may take
$Y_i:=E_i+a_iE_n$ for some $a_i\in\R$. Note that $a_{n-1}=0$ by Lemma \ref{L:facts}\eqref{I:ei+aei+1}, so $E_{n-1}\in\h$. Moreover  $a_{n-2}\neq 0$, as otherwise we would have $E_{n-2},E_{n-1}\in\h$, contradicting Lemma \ref{L:facts}\eqref{I:eiei+1}. So there are no elements of degree $n-1$ or $n$ in $\hp$.

The rest of the proof is done by   induction. Assume that for some $k$ with $p+2 < k < n$, there are no elements of degree $k,\ldots,n$ in $\hp$, but there is some $Z_{k-1}\in\hp$ with $\deg(Z_{k-1})=k-1$. Since $Z_{k-1}$ and $Y_{k-1}$ are orthogonal, we have  $a_{k-1}\neq 0$ and $\<Z_{k-1}, E_n\>\neq 0$. Then from the orthogonality  of $\hp$ and $\h$ we obtain   $a_{k-2},a_{k-3}=0$ giving  $E_{k-2},E_{k-3}\in\h$, which is impossible, by Lemma \ref{L:facts}\eqref{I:eiei+1}.
\end{proof}


Let $\mathcal O_1$ be the family of $n$-dimensional $\N$-graded  filiform Lie algebras for which the  basis $\B=\{X_1,\ldots,X_n\}$
may be chosen so that
\begin{equation}\label{E:O1}
\begin{split}
[X_1,X_i]& =X_{i+1}, \quad i=2,\ldots,n-1;\\
[X_i,X_{n-i}]&=\alpha_iX_n,\quad  (\alpha_i\neq 0),\quad i=2,\ldots, \lfloor (n-1)/2\rfloor.
\end{split}
\end{equation}

Additionally, denote by $\mathcal O_2$ the subfamily of $\mathcal O_1$ comprised of algebras for which $\B$
may be chosen so that conditions \eqref{E:O1} and condition
\begin{equation}\label{E:O2}
[X_2,X_i]=\beta_iX_{i+2},\quad (\beta_i\neq 0) \quad i=3,\ldots,n-2
\end{equation}
are satisfied.

\begin{lemma}\label{L:Xnnotinh}
Suppose that $X_n\not\in\h$ and that $n\geq 5$.
{\ }

\begin{enumerate}[\rm (a)]
\item \label{I:2.1a} If $\g$ belongs to $\mathcal O_1$, then $\dim(\h)\leq \lfloor n/2\rfloor$,
\item \label{I:2.1b} If $\g$ belongs to  $\mathcal O_2$, then $\dim(\h) \leq \lfloor (n -1)/2\rfloor$.
\end{enumerate}
\end{lemma}
\begin{proof}
(a) Suppose $\g$ belongs to $\mathcal O_1$. Since $X_n\not\in\h$, there is no elements of $\h$ of degree $n$ and if there is an element of $\h$ of degree $k$, then as $\h$ is a subalgebra, there are no elements of  $\h$ of degree $n-k$ except if $n$ is even and $k=\frac n2$.
Therefore, if $n$ is odd, $\dim(\h)\leq \frac{n-1}2$, and if $n$ is even,  $\dim(\h)\leq \frac{n}2$. This proves (a).

(b) Suppose $\g$ belongs to $\mathcal O_2$. From the proof of (a), we may assume that $n$ is even and that $\h$ contains an element $Y_{\frac n2}$ of degree $\frac n2$.
First suppose there is an element $Y_2\in\h$ with $\deg(Y_2)=2$. Let $m:=\frac n2$. If $m$ is even, then $U=\ad^{\frac m2}(Y_2)(Y_{m})\in\h$ and $\deg(U)=n$ contradicting the assumption. If $m$ is odd and $n\neq 6$, then since $\h$ would contain either an element $Y_{m -1}$  of degree $m -1$ or an element $Y_{m +1}$ of degree $m +1$, we would obtain either $V_1=\ad^{\frac {m-1}2+1}(Y_2)(Y_{m-1})\in\h$ and $\deg(V_1)=n$ or $V_2=\ad^{\frac {m-1}2}(Y_2)(Y_{m+1})\in\h$ and $\deg(V_2)=n$, respectively. In each case, we  get a contradiction with the assumption.
Let us discuss the case $n=6$. By Lemma \ref{L:facts}\eqref{I:e1}, we have $E_1\in\h^\perp$. Then from the above argument, we may assume that $\h=\Span(Y_2,Y_3,Y_5)$ for some $Y_i$ such that $\deg(Y_i)=i$. By Lemma \ref{L:facts}\eqref{I:ei+aei+1}, we have $\langle E_6,Y_5\rangle=0$. If there is $Z_2\in\h^\perp$ with $\deg(Z_2)=2$, then
\begin{equation*}
2\langle\nabla_{Y_3}Y_5,Z_2\rangle=\langle[Z_2,Y_3],Y_5\rangle\not=0,
\end{equation*}
which is impossible by \eqref{L:basic}. If there is $Z_3\in\h^\perp$ with $\deg(Z_3)=3$, then
\begin{equation*}
0=2\langle\nabla_{Y_2}Y_5,Z_3\rangle=\langle[Z_3,Y_2],Y_5\rangle,
\end{equation*}
is also impossible. The remaining case is $\h=\Span(E_2,E_3,E_5)$, which is   impossible by Lemma \ref{L:facts}\eqref{I:eiei+1}.

If there are no elements of degree 2 in $\h$ then $E_2\in\hp$ and in order to be a totally geodesic subalgebra of dimension $\frac {n}2$, $\h$ has to contain an element of degree $n-1$ as well as an element of degree $n-2$, say $Y_{n-1}$ and $Y_{n-2}$ respectively. By \eqref{L:basic}, conditions
\begin{equation*}
\langle\nabla_{Y_{n-1}}Y_{n-1},E_1\rangle=0\quad  \text {and }\quad \langle\nabla_{Y_{n-2}}Y_{n-2},E_2\rangle=0
\end{equation*}
give
\begin{equation*}
\langle X_n,Y_{n-1}\rangle=0\quad  \text {and }\quad \langle X_n,Y_{n-2}\rangle=0.
\end{equation*}
But then $E_{n-1},E_{n-2}\in\h$, contradicting Lemma \ref{L:facts}\eqref{I:eiei+1}.
\end{proof}

\begin{lemma}\label{L:g/Xn}
Let $\g$ be a Lie algebra with an inner product $\ip$ and let $\h \subset \g$ be a totally geodesic subalgebra.

\begin{enumerate}[\rm (a)]
  \item \label{I:g/Xna}
  If $\g$ is nilpotent, then the projection of its center to $\h$ lies in the center of $\h$.

  \item \label{I:g/Xnb}
  Let $\mathfrak{i} \subset \g$ be an ideal. Consider the Lie algebra $\overline{\g}:=\g/\mathfrak{i}$ and the quotient map $\pi:\g\to\overline{\g}$. Suppose that $\mathfrak{i}=(\mathfrak{i} \cap \h) \oplus (\mathfrak{i} \cap \h^\perp)$. Then there is an inner product on $\overline{\g}$ for which $\overline{\h}:=\pi(\h)$ is a totally geodesic subalgebra.
\end{enumerate}
\end{lemma}

\begin{proof} (a) Let $Z$ be a vector from the center of $\g$ and let $Z=Z_\h + Z_\perp, \; Z_\h \in \h, Z_\perp \in \h^\perp$. By \eqref{L:basic}, $\<[Z_\perp, X], X\>= 0$, for all $X \in \h$, hence $\<[Z_\h, X], X\>= 0$, for all $X \in \h$. As $\h$ is a subalgebra, it is an invariant subspace of the operator $\ad(Z_\h)$. Moreover, the restriction of $\ad(Z_\h)$ to $\h$ is both nilpotent and skew-symmetric, hence is zero, so $[Z_\h,\h]=0$.

(b) Let $\h'$ and $\h'_\perp$ be the orthogonal complements to $\mathfrak{i}$ in $\h$ and in $\h^\perp$ respectively and let $\g'=\h' \oplus \h'_\perp$. We equip the linear space $\g'$ with the inner product induced from that on $\g$ and with a bilinear skew-symmetric map $[\cdot,\cdot]'$ defined by $[X,Y]'=\pi'([X,Y])$, where $\pi':\g \to \g'$ is the orthogonal projection. As $\mathfrak{i}$ is an ideal, $[\cdot,\cdot]'$ is a Lie bracket, which turns $\g'$ into a Lie algebra isomorphic to $\overline{\g}$, with $\h' \subset \g'$ a subalgebra isomorphic to $\overline{\h}$. Moreover, as $\pi'(\h)=\h'$ and $\pi'(\h^\perp)=\h'_\perp$, the validity of condition \eqref{L:basic} for the vectors from $\h'$ and $\h'_\perp$ follows from that for the vectors from $\h$ and $\h^\perp$. \qedhere
\end{proof}

\begin{lemma}\label{L:Xninh}
Suppose that $X_n\in\h$ and that $n\geq 6$.

\begin{enumerate}[\rm (a)]
\item \label{I:Xninha} If $n$ is even and $\g/\Span(X_n)\in \mathcal O_1$, then $\dim(\h)\leq n/2$,
\item \label{I:Xninhb} If $n$ is odd and $\g/\Span(X_n)\in \mathcal O_2$, then $\dim(\h) \leq  (n -1)/2$.
\end{enumerate}
\end{lemma}
\begin{proof} Consider the quotient map $\pi:\g\to\overline{\g}:=\g/\Span(X_n)$.
By Lemma \ref{L:g/Xn}\eqref{I:g/Xnb}, there is an inner product on $\overline{\g}$ for which $\overline\h:=\pi(\h)$  is a totally geodesic subalgebra.  Let $\overline X_i:=\pi(X_i)$ for all $i=1,\ldots,n-1$. Note that $\overline X_{n-1}\not\in\overline\h$ since otherwise we would have $X_{n-1},X_n\in\h$, contradicting Lemma \ref{L:facts}\eqref{I:eiei+1}.
Hence if $n$ is even and $\overline{\g}\in \mathcal O_1$, then $\dim(\overline\h)\leq \lfloor (n-1)/2\rfloor=(n-2)/2$ by Lemma \ref{L:Xnnotinh}\eqref{I:2.1a}, so $\dim(\h)\leq n/2$.
Similarly, if $n$ is odd and $\overline{\g}\in \mathcal O_2$, then $\dim(\overline\h)\leq \lfloor (n-3)/2\rfloor=(n-3)/2$ by Lemma \ref{L:Xnnotinh}\eqref{I:2.1b}, so $\dim(\h)\leq  (n -1)/2$.
\end{proof}


\section{Algebras from Table \ref{Table:1}}\label{S:table1}
 We treat the algebras in the order they appear in Table \ref{Table:1}. First recall that by  \cite[Theorem 1.17]{CHGN}, for all $n\geq 3$, the Lie algebra $\m_0(n)$ possesses
an inner product relative to which $\m_0(n)$ has a totally geodesic subalgebra of
codimension two. This result is optimal since by  \cite[Proposition 1.13]{CHGN}, filiform Lie algebras have no totally geodesic subalgebras of
codimension one.

 \begin{remark}\label{R:codim}
 With the exception of $\m_0(n)$, the algebras of  Tables \ref{Table:1} and  \ref{Table:2}  have no totally geodesic subalgebras of codimension one, by  \cite[Proposition 1.13]{CHGN}, and none of codimension two, by \cite[Theorem 1.18]{CHGN}.
 \end{remark}

\begin{theorem}\label{T:m2n}
 If $\h$ is a proper totally geodesic subalgebra of $\m_2(n)$, then $\dim(\h)\leq \frac{n}2$.
\end{theorem}

\begin{proof}  Suppose that $\m_2(n)$ has an inner product $\ip$ for which  $\h$ is a totally geodesic subalgebra and assume that  $\dim(\h)> \frac{n}2$. First assume   that there are no elements in $\h$ of degree two. So there exist elements of $\hp$ of the following form: $Z=X_1+ \sum_{i=3}^n a_iX_i$ and $W=X_2+ \sum_{i=3}^n b_iX_i$, for some $a_i,b_i\in\R$. Note that $\h$ is contained in the derived algebra $[\m_2(n),\m_2(n)]$ of $\m_2(n)$. The key observation is that when restricted to $[\m_2(n),\m_2(n)]$, one has $\ad(W)=\ad^2(Z)$.

Let $\k=\hp\cap[\m_2(n),\m_2(n)]$ and consider the map $f: \h\to \k$ defined by $f(Y)= \pip([Z,Y])$, where $\pip:\m_2(n)\to\k$ is the orthogonal projection. We have $\dim(\h)> \frac{n}2= \frac{\dim(\h)+\dim(\k)+2}2$, so $\dim(\h)>\dim(\k)+2$. Thus  $\dim(\ker(f))>2$. Since $\dim(\ker(\ad(Z)_{|\h}))=1$, there exists $Y\in\h$ with $Y\in  \ker(f)\backslash \ker(\ad(Z))$; that is, $[Z,Y]\in\h$ and $[Z,Y]\not=0$. Since $\h$ is totally geodesic, $Z\in\hp$ and $Y,[Z,Y]\in\h$, we have
\[
0=\langle Y,[Z,[Z,Y]] \rangle+\langle [Z,Y],[Z,Y] \rangle
\]
so $\langle Y,[Z,[Z,Y]]\rangle\not=0$. But $\langle Y,[Z,[Z,Y]]\rangle=\langle Y,[W,Y]\rangle=0$, since $\h$ is totally geodesic, $W\in\hp$ and $Y\in\h$. This is a contradiction.

It remains to consider the case where  there exists $Y_2\in\h$ with $\deg(Y_2)=2$. First suppose that $E_n\not\in\h$.
If $n$ is odd then, since $\h$ has an element of degree 2, $\h$ can have no elements of odd degree and so $\dim (\h)\leq \frac{n-1}2$. If $n$ is even then $\h$ can have no elements of even degree $\geq 4$, and consequently $\dim (\h)\leq \frac{n}2$.

Now  suppose that  $E_n\in\h$. Note that by Lemma \ref{L:facts}\eqref{I:eiei+1}, $\h$ has no elements of degree $n-1$.
Hence, if $n$ is even, $\h$ can have no elements of odd degree, and so $\dim (\h)\leq \frac{n}2$.
 If $n$ is odd, $\h$ can have no elements of even degree $\geq 4$, and consequently $\dim (\h)\leq \frac{n+1}2$.
Suppose therefore that $n$ is odd and that $\h$ has dimension $ \frac{n+1}2$; so there are elements $Y_i$ for $i=3,5,7,\dots, n$ such that $Y_i$ has degree $i$ and $\h=\Span(Y_2,Y_3,Y_5,\dots,Y_n)$. Without loss of generality, we may assume that for each $i=3,5,7,\dots, n$, the vector $Y_i$ has no component in the $E_j$ direction for all odd $j>i$.

Note that $\hp$ has no elements of degree 2. Indeed, if $W$ were such an element, then $[W,Y_{n-2}]$ would have degree $n$, but we would also have
\[
 0=\langle Y_{n-2},[W,Y_{n}] \rangle+ \langle [W,Y_{n-2}],Y_{n} \rangle=  \langle [W,Y_{n-2}],Y_{n} \rangle,
 \]
which is impossible.

We claim that $\hp$ has no elements of odd degree $\geq 3$. Indeed, suppose that $\hp$ has an element $W$ of odd degree $i\geq 3$. Note that $[W,Y_{2}]$ has degree $i+2$ and for each odd $j\geq 3$, we have $Y_j,W\in\Span(X_3,X_4,\dots,X_n)$ and so $[W,Y_{j}]=0$. So
 \[
 0=\langle Y_{2},[W,Y_{j}] \rangle+ \langle [W,Y_{2}],Y_{j} \rangle=  \langle [W,Y_{2}],Y_{j} \rangle.
 \]
Moreover, $\langle Y_{2},[W,Y_{2}] \rangle=0$. Hence $[W,Y_{2}]\in\hp$. By induction, we obtain an element of $\hp$ of degree $n$, contradicting the assumption that $E_n\in\h$.

From what we have just seen,  since $\dim(\hp)=\frac{n-1}2$, there are necessarily elements $W_i$ for $i=4,6,8,\dots, n-1$ such that $W_i$ has degree $i$ and $\hp=\Span(E_1,W_4,W_6,\dots,W_{n-1})$. We may assume that for each $i=4,6,8,\dots, n-1$, the vector $W_i$ has no component in the $E_j$ direction for all even $j>i$. Note that as $W_{n-1}$ has degree ${n-1}$, we have $W_{n-1}\in\Span(E_{n-1},E_n)$. So, as $W_{n-1}$ is perpendicular to $Y_n$, which is a multiple of $E_n$, we must have that $W_{n-1}$ is  a multiple of $E_{n-1}$. Continuing by induction, it is clear that  $W_i$  is a multiple of $E_i$ for all  $i=4,6,8,\dots, n-1$, and $Y_i$  is a multiple of $E_i$ for all   $i=2,3,5,7,\dots, n$. But then as $Y_2,Y_3\in\h$ we have $E_2,E_3\in\h$, which contradicts Lemma~\ref{L:facts}\eqref{I:eiei+1}. This completes the proof of the theorem.
\end{proof}


\begin{theorem}\label{T:dimmiracle}
If $\h$ is a proper totally geodesic subalgebra of $\V_n$, $n \ge 3$, then $\dim(\h)\leq \frac{n}2$.
 \end{theorem}

\begin{proof}
By  \cite[Proposition 1.13]{CHGN}, filiform Lie algebras have no totally geodesic subalgebras of codimension one. So the result is true for $n\leq 4$. For $n = 5$ the claim follows from Theorem~\ref{T:m2n} and the isomorphism from Remark~\ref{R:DMilTheorem}. Suppose $n \ge 6$. By Lemma~\ref{L:Xnnotinh} we may assume that $X_n\in\h$ and then the required result follows immediately from Lemma~\ref{L:Xninh}, as $\V_n/\Span(X_n)\cong \V_{n-1} \in \mathcal O_2$.
\end{proof}


\begin{theorem}\label{T:m0,1}
The minimal codimension of a proper totally geodesic subalgebra $\h$ of $\m_{0,1}(2k+1)$, $k \ge 3$, is four.
\end{theorem}
\begin{proof}
Let $\h$ be a proper totally geodesic subalgebra of $\m_{0,1}(2k+1)$, $k\geq 3$. If $X_{2k+1}\not\in\h$,
then Lemma \ref{L:Xnnotinh}\eqref{I:2.1a} gives $\dim (\h)\leq k$, since $\m_{0,1}(2k+1)\in\mathcal O_1$ for every $k\geq 3$.
So we may assume that $X_{2k+1}\in\h$.
Hence, by Lemma \ref{L:facts}\eqref{I:eiei+1}, there are no elements of degree $2k$ in $\h$.
By Lemma \ref{L:facts}\eqref{I:e1}, $\h\subset \Span(X_2,\ldots,X_{2k+1})$.
Moreover, by \cite[Proposition~1.13, Theorem~1.18]{CHGN}, the codimension of $\h$ is at least $3$. If the codimension of $\h$ is $3$, there exists $2 \le i \le 2k-1$ such that $\h=\Span(Y_2, \dots, Y_{i-1},Y_{i+1}, \dots, Y_{2k-1}, X_{2k+1})$, where $\deg (Y_j)=j$. We may then choose a basis for the subspace $\g'=\h \oplus \Span(X_i) \subset \g$ of the form $\{X_j + a_jX_{2k}, X_{2k+1}\}, \; j=2, \dots, 2k-1$, for some $a_j \in \R$, hence $\g'$ is a Lie algebra isomorphic to the Heisenberg algebra. It follows from \eqref{L:basic} that $\h$ is totally geodesic in $\g'$, with the induced inner product, but this is a contradiction with \cite[Proposition~1.13]{CHGN}.
%
Hence $\h$ has codimension at least 4 in $\m_{0,1}(2k+1)$.

In the following, we will exhibit an example of a totally geodesic subalgebra $\h$ of $\m_{0,1}(2k+1)$ of codimension exactly four.
Fix $k \ge 3$ and let $\m=\R^{2k+1}$, equipped with an inner product $\ip$ and an orthonormal basis $\{E_1,\dots,E_{2k+1}\}$. Introduce the subspace $\m'=\Span(E_2, \dots , E_{2k+1})$, and
define a bilinear skew-symmetric map $[\cdot,\cdot]: \m \times \m \to \m$ by
\begin{equation}\label{eq:m01KN}
[E_1, X] = NX, \quad [X,Y]=\<KX,Y\>E_{2k+1}, \quad \text{for all } X, Y \in \m',
\end{equation}
where the operators $N, K \in \End(\m')$ are defined by their matrices relative to the basis $\{E_2, \dots, E_{2k+1}\}$ for $\m'$ as follows:
\begin{equation}\label{eq:NKmat}
    N=\left(
        \begin{array}{cc|cc}
          0 & S & 0 & 0 \\
          -S +uu^t& 0 & 0 & 0 \\
          \hline
          p^t & 0 & 0 & 0 \\
          0 &  0  & 1 & 0 \\
        \end{array}
      \right), \qquad
    K=\left(
        \begin{array}{cc|cc}
          0 & I_{k-1} & 0 & 0  \\
          -I_{k-1} & 0 & 0 & 0 \\
          \hline
          0  & 0 & 0 & 0 \\
          0  & 0 & 0 & 0 \\
        \end{array}
      \right),
\end{equation}
where $I_{k-1}$ is the identity matrix, 
$u, p \in \R^{k-1}$, and $S$ is a symmetric nonsingular $(k-1)\times (k-1)$-matrix such that the matrix
\begin{equation}\label{eq:SSEnil}
T=S(-S + uu^t) \text{ is nilpotent}. 
\end{equation}
We postpone the question of the existence of such $S$ and $u$ and of a correct choice of $p$ to a little later.

The map $[\cdot,\cdot]$ given by \eqref{eq:m01KN} can be extended to $\m$ by skew symmetry and bilinearity (note that $K$ is skew-symmetric).
The claim of the theorem is established in the following four steps:
\begin{enumerate}[(i)]
  \item \label{it:i}
  The space $\m$ with the map $[\cdot,\cdot]$ defined by (\ref{eq:m01KN}, \ref{eq:NKmat}) is a Lie algebra. 
 \item \label{it:ii}
 The subspace $\h = \big(\Span(E_1, E_{2k}, (0, u, 0_{k+1})^t, (0_{k}, u, 0,0)^t)\big)^\perp$ is a totally geodesic subalgebra of $\m$, where $0_m$ denotes the row vector of $m$ zeros.
  \item \label{it:iii}
 There exists $u \in \R^{k-1}$ and a symmetric nonsingular matrix $S$ satisfying \eqref{eq:SSEnil}.
  \item \label{it:iv}
 There exists $p \in \R^{k-1}$ such that the Lie algebra $\m$ defined by (\ref{eq:m01KN}, \ref{eq:NKmat}, \ref{eq:SSEnil}), with $S$ and $u$ constructed as in \eqref{it:iii}, is isomorphic to $\m_{0,1}(2k+1)$. 
\end{enumerate}

\eqref{it:i}.  To see this, it suffices to check the Jacobi identities. As $[\m, E_{2k+1}]=0$ and $[\m',\m'] = \Span(E_{2k+1})$, they are satisfied for any triple of vectors from $\m'$. By \eqref{eq:m01KN}, the Jacobi identity on a triple  $(E_1,X,Y), \;  X,Y\in \m'$, is equivalent to $\<KNX,Y\>=\<KNY,X\>$, for all $X,Y \in \m'$, which is true, as $KN$ is symmetric (from \eqref{eq:NKmat}).

\eqref{it:ii}.   Note that by \eqref{eq:SSEnil}, $u \ne 0$ as $S$ is nonsingular, so $\mathrm{codim}(\h)=4$.
The subspace $\h$ is a subalgebra since $[\h,\h] \subset [\m',\m'] \subset \Span(E_{2k+1}) \subset \h$. To see that $\h$ is totally geodesic, we have to check that \eqref{L:basic} is satisfied. If $X \in \m'$, equation \eqref{L:basic} is equivalent to $\<KX,Y\>\<E_{2k+1},Z\>+\<KX,Z\>\<E_{2k+1},Y\>=0$. As $K$ is skew-symmetric, this is equivalent to $\<X,\<E_{2k+1},Z\>KY+\<E_{2k+1},Y\>KZ\>=0$, which is true since $KY,KZ \in \h$, as $\h \subset \m'$ is $K$-invariant (see \eqref{eq:m01KN}).

If $X=E_1$, then by \eqref{eq:m01KN}, condition \eqref{L:basic} is equivalent to $\<NY,Z\>+\<NZ,Y\>=0$, for all $Y, Z \in \h$, which follows from the form of $N$ given in \eqref{eq:m01KN} and the definition of $\h$.

\eqref{it:iii}.  Suppose that a nonsingular symmetric operator $S$ and a vector $u$ satisfy
  \begin{equation}\label{eq:Su}
  \begin{gathered}
    \rk(S^{-1}u, S^{-3}u, \dots, S^{3-2k}u)=k-1, \\
    \<S^{-1}u,u\>=1, \quad \<S^{-3}u,u\>= \dots =\<S^{3-2k}u,u\>=0.
  \end{gathered}
  \end{equation}
  Then by the second condition of \eqref{eq:Su}, $TS^{-1}u=S(-S+uu^t)S^{-1}u=0$, and then $T^2S^{-3}u=TS(-S+uu^t)S^{-3}u=-TS^{-1}u=0$, and, by induction, $T^{j}S^{1-2j}u=0$, for all $j=1, \dots, k-1$. So
  \begin{equation}\label{eq:TmSju}
    T^m S^{1-2j}u=0, \text{ for all } 1 \le j \le m \le k-1,
  \end{equation}
  and in particular, $T^{k-1}S^{-1}u=T^{k-1}S^{-3}u= \dots =T^{k-1}S^{3-2k}u=0$. As by the first condition of \eqref{eq:Su} the vectors $S^{-1}u, \dots, S^{3-2k}u$ form a basis for $\R^{k-1}$, we obtain $T^{k-1}=0$, as required.

  To construct $S$ and $u$ satisfying \eqref{eq:Su}, consider a diagonal matrix $S$ with distinct positive diagonal entries $d_i$, for $i=1, \dots, k-1$, and a vector $u=(u_1,\dots,u_{k-1})^t$ none of whose entries are zero. The first condition of \eqref{eq:Su} is equivalent to the condition that the $(k-1)$ $\times (k-1)$ matrix $M$ with entries $M_{ij}=d_i^{1-2j}u_i$ is nonsingular. As $\det (M)= \prod_i (d_i^{-1}u_i) \times V(d_1^{-2}, \dots, d_{k-1}^{-2})$, where $V$ is the Vandermonde determinant, the first condition is satisfied since $d_i^2 \ne d_j^2$ for $i \ne j$. The second condition of \eqref{eq:Su} is equivalent to the  following condition:
  \begin{equation}\label{eq:seccond}
  \sum_{i=1}^{k-1} d_i^{1-2l}u_i^2=\begin{cases}
1&:\ \text{if}\  l=1,\\
0&:\ \text{otherwise},
\end{cases}
 \end{equation}
for $l=1, \dots, k-1$.
To obtain this,  we choose the $u_i$'s and adjust the signs of the $d_i$'s in such a way that $d_i^{-1}u_i^2=\prod_{j \ne i}d_i^2(d_i^2-d_j^2)^{-1}$, for $i=1, \dots, k-1$. To see that this works, we employ the following combinatorial result:

\begin{lemma}\label{L:comb}
Let $b_1,\dots,b_m$ be distinct nonzero reals, where $m\geq 2$. Then for $ 0\leq l \leq m-1$,
\[
\sum_{i=1}^m b_i^l\prod_{j\not=i} \frac1{b_i-b_j}=
\begin{cases}
1&:\ \text{if}\  l=m-1,\\
0&:\ \text{otherwise}.
\end{cases}
\]
\end{lemma}

\begin{proof} Let $\phi_i(t)= \prod_{j \ne i}(t-b_j^{-1})(b_i^{-1}-b_j^{-1} )^{-1}$. So $\phi_i(b_j^{-1})=\delta_{ij}$ for all $1\leq i,j\leq m$.
Consider the polynomial $f_l(t)=\sum_{i=1}^m b_i^{l-m+1} \phi_i(t)$. Note that $f_l(t)$ has degree $m-1$ and $f_l(b_i^{-1})=b_i^{-(m-l-1)}$ for each $i$. So $f_l(t)=t^{m-l-1}$. So
\begin{align*}
 \sum_{i=1}^m b_i^l\prod_{j\not=i} \frac1{b_i-b_j}&=
 \sum_{i=1}^m b_i^{l-m+1} \prod_{j \ne i}\frac{b_i }{b_i -b_j }\\
 &=
 \sum_{i=1}^m b_i^{l-m+1} \prod_{j \ne i}\frac{-b_j^{-1} }{b_i^{-1} -b_j^{-1} }=
f_l(0)=\delta_{l,m-1}.\qedhere
 \end{align*}
\end{proof}
Setting $m=k-1$ and $b_i= d_i^2$, the lemma gives
\begin{align*}
 \sum_{i=1}^{k-1} d_i^{1-2l}u_i^2= \sum_{i=1}^{k-1} d_i^{2-2l}\prod_{j \ne i}d_i^2(d_i^2-d_j^2)^{-1}&= \sum_{i=1}^{k-1} d_i^{2k-2-2l}\prod_{j \ne i}(d_i^2-d_j^2)^{-1}\\
&=\sum_{i=1}^m b_i^{m-l}\prod_{j\not=i} \frac1{b_i-b_j}=\delta_{1l},
\end{align*}
which gives \eqref{eq:seccond} as required. This completes (iii).

\eqref{it:iv}.  We first consider the requirement that $\m$ be filiform, with $E_1$ having maximal nilpotency. Note that $\m'$ contains the derived algebra of $\m$ and is isomorphic to the direct product of the Heisenberg algebra $Heis(2k-1)$ and the one-dimensional abelian ideal spanned by $E_{2k}$ (and hence is nilpotent). It therefore suffices to show that the restriction of $\ad(E_1)$ to $\m'$ is nilpotent and has a maximal rank, that is, that $N$ is a nilpotent  matrix and that $N^{2k-1} \ne 0$. For $m \ge 1$, we have by induction:
  \begin{equation}\label{eq:N^2m+1}
    N^{2m+1}=\left(
        \begin{array}{cc|cc}
          0 & S(T^t)^m & 0 & 0 \\
          (-S +uu^t)T^m& 0 & 0 & 0 \\
          \hline
          p^t T^m & 0 & 0 & 0 \\
          0 & p^t S(T^t)^{m-1} & 0 & 0 \\
        \end{array}
      \right).
  \end{equation}
  The $(k-1) \times (k-1)$ matrix $T$ is nilpotent by \eqref{eq:SSEnil}, so $T^{k-1}=0$, so $N$ (and hence $\m$) is indeed nilpotent. Moreover, for $m=k-1$ we have
    \begin{equation}\label{eq:N^2k-1}
    N^{2k-1}=\left(
        \begin{array}{cc|cc}
          0 & 0 & 0 & 0 \\
          0& 0 & 0 & 0 \\
          \hline
          0 & 0 & 0 & 0 \\
          0 & (T^{k-2}Sp)^t & 0 & 0 \\
        \end{array}
      \right),
  \end{equation}
  so $\m$ is filiform if and only if
  \begin{equation}\label{eq:mnilp}
  T^{k-2}Sp \ne 0.
  \end{equation}
  Let us assume for the moment that \eqref{eq:mnilp} is satisfied.
   To show that $\m$ is isomorphic to $\m_{0,1}(2k+1)$, we have to find a vector $X_2 \in \m'$ such that the vectors $X_{i}=N^{i-2}X_2, i = 2, \dots, 2k+1$, are nonzero and satisfy $[X_i, X_j]=(-1)^{i+1} \delta_{i+j, 2k+1} X_{2k+1}$, for all $i,j=2, \dots, 2k+1$. We choose the vector $X_2$ whose coordinates relative to the basis $\{E_2, \dots, E_{2k+1}\}$ for $\m'$ are $(X_2)^t=( 0_{k-1},q^t,0,0)$, where  $q \in \R^{k-1}$ satisfies
  \begin{equation}\label{eq:pq}
  \<T^{k-2}Sp,q\> =-1. 
  \end{equation}
  Then $X_{2k+1}=N^{2k-1}X_2=-E_{2k+1}$ by \eqref{eq:N^2k-1}, so by \eqref{eq:m01KN}, the vector $q$ has to be chosen in such a way that $\<KN^{i-2}X_2,N^{j-2}X_2\>=(-1)^{i} \delta_{i+j, 2k+1}$, for all $i,j=2, \dots, 2k+1$. Note that the matrix $KN$ is symmetric, so $N^tK=-KN$. It follows that if $i+j$ is even (say $i+j=2s$), then $\<KN^{i-2}X_2,N^{j-2}X_2\>=\pm \<KN^{s-2}X_2,N^{s-2}X_2\>=0$, as $K$ is skew-symmetric. If $i+j=2s+1$ is odd, we have $\<KN^{i-2}X_2,N^{j-2}X_2\>=(-1)^{j-2}\<KN^{2s-3}X_2,X_2\>=(-1)^{i+1}\<KN^{2s-3}X_2,X_2\>$. So we require $\<KN^{2s-3}X_2,X_2\>= -\delta_{sk}$, for all $s=2, \dots, k$. By \eqref{eq:N^2m+1} and the choice of $X_2$, the latter equation is equivalent to $\<S(T^t)^{m}q,q\>= \delta_{m,k-2}$, for all $m=0, \dots, k-2$ (note that $T^{k-1}=0$ anyway). As $ST^t=TS$, this is equivalent to
  \begin{equation}\label{eq:TmS}
    \<T^{m}Sq,q\>= \delta_{m,k-2}, \text{ for all } m=0, \dots, k-2.
  \end{equation}
  Define $q=S^{2-2k}P(S^2)u$, where $P(t)=a_0+a_1t+ \dots + a_{k-2}t^{k-2}$ is a polynomial which we will specify a little later. By \eqref{eq:TmSju}, $T^mS^{1-2i}u = 0$, for $1 \le i \le m \le k-1$. Moreover, since $T=S(-S+uu^t)$ one obtains by induction using \eqref{eq:Su} that $T^mS^{1-2i}u = (-1)^mS^{2m+1-2i}u$, for $0 \le m < i  \le k-1$. It follows that $\<T^m S^{1-2i}u, S^{-2j}u\>=0$, when $1 \le i \le m \le k-1$, and $\<T^m S^{1-2i}u, S^{-2j}u\>=(-1)^m\< S^{2m+1-2i-2j}u, u\>$, when $0 \le m < i  \le k-1$, so, again by \eqref{eq:Su}, $\<T^m S^{1-2i}u, S^{-2j}u\>=0$, when $i+j < k+m$, and $\<T^m S^{1-2i}u, S^{-2j}u\>=(-1)^m\< S^{2m+1-2i-2j}u, u\>$, if $k+m \le i+j \le 2(k-1)$. It follows that
  \begin{align*}
    \<T^{m}Sq,q\>&=\sum\nolimits_{s,l=0}^{k-2}a_sa_l\<T^{m}S^{3-2k+2s}u,S^{2-2k+2l}u\>\\
    &=(-1)^{m}\sum\nolimits_{s,l;0 \le s+l \le k-m-2}a_sa_l\<S^{2m+5-4k+2s+2l}u,u\>\\
    &=(-1)^{m}\sum\nolimits_{M=0}^{k-m-2}\<S^{2m+5-4k+2M}u,u\> \big(\sum\nolimits_{s,l \ge 0; s+l=M} a_sa_l\big) \\
    &=(-1)^{m}\sum\nolimits_{M=0}^{k-2}\<S^{2m+5-4k+2M}u,u\> [P^2]_M, \\
  \end{align*}
  where $[P^2]_M$ is the coefficient of $t^M$ in the polynomial $P^2(t)$, and where we changed the upper limit in the second summation from $k-2$ to $k-m-2$, since for $m \le k-2$ and $k-m-1 \le M \le k-2$, we have $3-2k \le 2m+5-4k+2M \le 2m+1-2k \le -3$, so $\<S^{2m+5-4k+2M}u,u\>=0$ by \eqref{eq:Su}. It follows that
  \begin{equation}\label{eq:[p^2]}
    \<T^{m}Sq,q\>=(-1)^{m}\<S^{2m+5-4k}\sum\nolimits_{M=0}^{k-2}[P^2]_MS^{2M}u,u\>,
  \end{equation}
  To choose the polynomial $P$, we need the following lemma.
  {
  \begin{lemma}\label{l:square}
  Let $\chi(t)$ be a polynomial of degree $r$ with a positive constant term. Then there exists a polynomial $P(t)$ of degree $r$, such that $t^{r+1} \mid (P^2(t) - \chi(t))$.
  \end{lemma}
  \begin{proof} Let $\chi(t)=a^2 + b_1t+ \dots + b_rt^r, \; a \ne 0$. Formally, $P(t)$ is just the truncation of the formal power series $a(1+(a^{-2}b_1t+ \dots + b_rt^r))^{1/2}$ up to the term $t^r$. Informally, for $P(t)=c_0 + c_1t+ \dots + c_rt^r$, we have $c_0=a$, and then $2c_0c_1=b_1$, which can be solved for $c_1$; then $2c_0c_2+c_1^2=b_2$, which can be solved for $c_2$, and so on.
  \end{proof}
  }
  Now  in Lemma~\ref{l:square}, take $r=k-2$ and $\chi(t)=\det(S^2-tI_{k-1})+(-1)^{k}t^{k-1}$ (note that $\deg (\chi)=k-2$), and choose the corresponding $P(t)$. Then $\sum_{M=0}^{k-2}[P^2]_MS^{2M}=\sum_{M=0}^{k-2}[\chi]_MS^{2M}=\chi(S^2)=(-1)^{k}S^{2k-2}$, by the Cayley--Hamilton theorem. Then from \eqref{eq:[p^2]} and from \eqref{eq:Su} we obtain that for all $m=0, \dots, k-2$,
  \begin{equation*}
    \<T^{m}Sq,q\>=(-1)^{m+k}\<S^{2m+3-2k}u,u\>=(-1)^{m+k}\K_{m,k-2}=\K_{m,k-2},
  \end{equation*}
  as required by \eqref{eq:TmS}. As by \eqref{eq:TmS}, the vector $(T^{m}S)^tq$ is clearly nonzero, we can find $p \in \R^{k-1}$ satisfying \eqref{eq:pq}, and then \eqref{eq:mnilp} will be satisfied automatically.

  This proves the last step, and hence, completes the proof of the theorem.
\end{proof}


\begin{theorem}\label{T:m0,2}
The maximal dimension of a proper totally geodesic subalgebra $\h$ of\break ${\m_{0,2}(2k+2)}$, $k\geq 3$, is $k+1$.
\end{theorem}

\begin{proof}
As $\m_{0,2}(2k+2) \in \mathcal O_1$, we have $\dim(\h) \leq k+1$ by Lemma \ref{L:Xnnotinh}\eqref{I:2.1a}, if $X_{2k+2}\not\in\h$. So we may suppose that  $X_{2k+2}\in\h$, in which case the claim follows from Lemma~\ref{L:Xninh}\eqref{I:Xninha}, as $\m_{0,2}(2k+2)/\Span(X_n)\cong \m_{0,1}(2k+1)\in \mathcal O_1$.
\end{proof}


\begin{theorem}\label{T:m0,3}
The maximal dimension of a proper totally geodesic subalgebra $\h$ of\break ${\m_{0,3}(2k+3)}$, $k\geq 3$, is $k+1$.
\end{theorem}

\begin{proof}
Let $\g=\m_{0,3}(2k+3), \; k \ge 3$, with the inner product $\ip$ and let $\h \subset \g$ be a proper totally geodesic subalgebra.
First note that if  $X_{2k+3}\in\h$, then we have  the  quotient map
\begin{equation*}
\pi:\g\to \overline{\g}:=\m_{0,3}(2k+3)/\Span(X_{2k+3})\cong \m_{0,2}(2k+2).
\end{equation*}
and by Lemma \ref{L:g/Xn}\eqref{I:g/Xnb}, there is an inner product on $\overline{\g}$ for which $\overline\h:=\pi(\h)$  is a totally geodesic subalgebra.
Then by Theorem \ref{T:m0,2}, we have $\dim (\overline\h) \leq k+1$ and so $\dim(\h)\leq k+2$.
On the other hand, if  $X_{2k+3}\not\in\h$, then for all $m = 3,\ldots,k+1$, since $[X_m, X_{2k-m+3}] $ is a nonzero multiple of $X_{2k+3}$, the subalgebra $\h$ cannot have an element of degree $m$ and an element of degree ${2k-m+3}$. Thus, using Lemma \ref{L:facts}\eqref{I:e1}, we again have $\dim(\h) \leq k+2$. So in either case, we have $\dim(\h) \leq k+2$, while we require $\dim(\h) \leq k+1$.

Consider the subalgebra $\g'=\Span(X_2,\dots,X_{2k+3})$.
Note that  $\g'$ is two-step nilpotent and that $\n:=[\g',\g']=\Span(X_{2k+1},X_{2k+2},X_{2k+3})$ is its center.
Let $\b$ denote the orthogonal complement of $\n$ in $\g'$; so $\b$ is a $(2k-1)$-dimensional vector subspace of $\g'$, but not a subalgebra.
For $N \in \n$, define a skew-symmetric operator $J_N \in \mathfrak{so}(\b)$ by $\<J_NX,Y\>=\<N,[X,Y]\>$, for $X,Y \in \b$. We have the following lemma, which roughly says that although the two-step algebra $\g'$ is not non-degenerate, its degeneracy can be tightly controlled.

\begin{lemma}\label{l:Kabc}
{\ }

\begin{enumerate}[\rm (a)]
  \item \label{it:Ka}
  If $N \in \n$ is nonzero, then $\rk (J_N) = 2k-2$.
  \item \label{it:Kb}
  If $N \in\n$ is nonzero and $V\subset\b$ is a subspace of maximal dimension such that $\<J_NV,V\>=0$, then $\dim (V)=k$ and $\ker (J_N) \subset V$.
  \item \label{it:Kc}
  If $N_1, N_2 \in \n$ are not proportional, then $\ker (J_{N_1}) \cap \ker (J_{N_2})=0$.
  \item \label{it:Kd}
  If $N_1, N_2 \in \n$ are not proportional, then the space $L=\Span(\ker (J_N) \, | \, N \in \Span(N_1,N_2),\break N \ne 0)$ has dimension $k$.
\end{enumerate}
\end{lemma}
\begin{proof}
The subspace $\b$ has a basis (possibly non-orthonormal) of the form $Y_i=X_i+Z_i, \; i=2, \dots, 2k$, where $Z_i \in \n$. Note that $[Y_i,Y_j] =[X_i,X_j]$, for all $i,j=2, \dots, 2k$. The matrix of the operator $J_N$ relative to the orthonormal basis $\{E_i\ :\  i=2, \dots, 2k\}$, for $\b$ is given by $J_N=\sum_{\al=1}^3 \<X_{2k+\al},N\> B K_\al B^t$, where $B$ is the transformation matrix between the bases $\{Y_i\}$ and $\{E_i\}$, and the skew-symmetric matrices $K_\al, \; \al=1,2,3$, are given by the defining relations of $\m_{0,3}(2k+3)$, relative to the basis $\{X_i\}$ (that is, $[X_i,X_j]=\sum_{\al=1}^3 (K_\al)_{i-1,j-1} X_{2k+\al}$, for $i,j=2, \dots, 2k$). Explicitly, they are given by
\begin{equation}\label{eq:K123}
\begin{gathered}
    (K_1)_{lm}=(-1)^l \delta_{l+m,2k-1}, \quad (K_2)_{lm}=(-1)^l \delta_{l+m,2k}(k-l), \\ (K_3)_{lm}=(-1)^{l+1} \delta_{l+m,2k+1} \frac12 (l-1)(m-1),
\end{gathered}
\end{equation}
for $l,m=1, \dots, 2k-1$. We have $J_N=B(aK_1+bK_2+cK_3)B^t$, where $a=\<X_{2k+1},N\>, \, b=\<X_{2k+2},N\>, \, c=\<X_{2k+3},N\>$. As $B$ is  nonsingular, it suffices to prove all four assertions of the lemma, with the $J_N$'s replaced by the matrices $aK_1+bK_2+cK_3$. From \eqref{eq:K123},
\begin{multline}\label{eq:Kabc}
    aK_1+bK_2+cK_3=\\
    \left(
        \begin{array}{ccccccc}
          0 & 0 & 0 & \hdots & 0 &-a & -(k-1)b \\
          0 & 0 & 0 & \hdots & a & (k-2)b & -(k-1)c \\
          0 & 0 & 0 & \iddots & -(k-3)b & (2k-3)c & 0 \\
          \vdots & \vdots &\iddots &\iddots &\iddots &\vdots & \vdots\\
          0 & -a& (k-3)b & \iddots & 0 & 0 & 0 \\
          a & -(k-2)b & -(2k-3)c & \hdots & 0 & 0 & 0 \\
          (k-1)b & (k-1)c & 0 & \hdots & 0 & 0 & 0 \\
        \end{array}
      \right)
\end{multline}

\begin{enumerate}[\rm (a)]
    \item
    As $aK_1+bK_2+cK_3$ is skew-symmetric and as $\dim (\b)=2k-1$ is odd, $\dim(\ker(aK_1+bK_2+cK_3))$ is odd. Assume that it is of dimension at least three. Suppose that $c \ne 0$. Then there exists a nonzero vector $x=(x_1, x_2, \dots, x_{2k-1})^t \in \ker(aK_1+bK_2+cK_3)$ whose first two coordinates are zero: $x_1=x_2=0$. Multiplying the matrix given by \eqref{eq:Kabc} by such an $x$ we obtain $x_3=0$ from the second last row, then $x_4=0$ from the third last row, and so on (note that all the entries on the sub-antidiagonal are nonzero by \eqref{eq:K123}). Then $x=0$, a contradiction. A similar argument (starting with the last two coordinates being zero) also works when $a \ne 0$. If $a=c=0, \; b \ne 0$, then from \eqref{eq:Kabc}, $\ker (K_2)$ is spanned by the $k$-th coordinate vector.

    \item
    As the subspaces $V$ and $J_NV$ are orthogonal and as $\dim(\ker (J_N)) =1$ by assertion~\eqref{it:Ka}, we have $2k-1 = \dim (\b) \ge \dim (V) + \dim (J_NV) \ge 2 \dim (V)-1$. It follows that $\dim (V) \le k$, with the equality only possible when $\dim (V)=k$ and $\dim (J_NV)=k-1$, which implies $\ker (J_N) \subset V$.

    \item
    By assertion~\eqref{it:Ka}, the kernel of any nonzero $K=aK_1+bK_2+cK_3$ has dimension one. We can find that kernel explicitly, namely
    \begin{equation}\label{eq:kerK}
        \begin{gathered}
        \ker (aK_1+bK_2+cK_3)= \Span(x), \text{ where } x=(x_1, \dots, x_{2k-1})^t,  \\ 
        \sum\nolimits_{j = 0}^{2k-2} (j!)^{-1}x_{2k-1-j}t^j =(a-bt+\tfrac12 c t^2)^{k-1}.
        \end{gathered}
    \end{equation}
    To prove this, assume that $x=(x_1, \dots, x_{2k-1})^t$ is a nonzero vector from $\ker (aK_1+bK_2+cK_3)$ and define $y_j=x_{2k-1-j}$ for $j=0, \dots, 2k-2$. Then from \eqref{eq:K123} we obtain $ay_{j+1}+b(k-j-1)y_j-\frac12 c j(2k-j-1)y_{j-1}=0$, for all $j=0, \dots, 2k-2$ (with the understanding that $y_{-1}=y_{2k-1}=0$). Multiplying by $t^j/j!$ and summing up by $j=0, \dots , 2k-2$ we obtain $ap'(t)+b(-tp'(t)+(k-1)p(t))+c(\frac12 t^2 p'(t)-(k-1)tp)=0$, where $p(t)=\sum_{j = 0}^{2k-2}y_j t^j/j!$. It follows that the polynomial $p(t)$ satisfies the equation $(a-tb+\frac12 ct^2)p'(t)=(k-1)(a-tb+\frac12 ct^2)'p(t)$, a solution to which is $(a-bt+\tfrac12 c t^2)^{k-1}$.
  It follows from \eqref{eq:kerK} that two nonzero matrices $a_1K_1+b_1K_2+c_1K_3$ and $a_2K_1+b_2K_2+c_2K_3$ have the same kernel only when the polynomials $(a_i-b_it+\tfrac12 c_i t^2)^{k-1}, \; i=1,2$, are proportional, which implies $(a_1,b_1,c_1) \parallel (a_2,b_2,c_2)$.

    \item
   From \eqref{eq:kerK}, the dimension of the span of the kernels of all the nontrivial linear combinations of nonproportional matrices $a_1K_1+b_1K_2+c_1K_3$ and $a_2K_1+b_2K_2+c_2K_3$ is the same as that of the subspace $S=\Span_{\lambda \in \R}(((a_1+\lambda a_2)-(b_1+\lambda b_2)t+\tfrac12 (c_1+\lambda c_2) t^2)^{k-1})$ in the space of polynomials in $t$ of degree less than or equal to $2k-2$. The subspace $S$ is spanned by the polynomials $(a_1-b_1t+\tfrac12 c_1 t^2)^{k-1-j}(a_2-b_2t+\tfrac12 c_2 t^2)^{j}, \; j=0, \dots , k-1$. If they were linearly dependent, then there would exist a nontrivial relation of the form $\sum_{j=0}^{k-1} \mu_j (a_1-b_1t+\tfrac12 c_1 t^2)^{k-1-j} (a_2-b_2t+\tfrac12 c_2 t^2)^{j}=0$. Dividing both sides by $(a_1-b_1t+\tfrac12 c_1 t^2)^{k-1}$ we obtain that a nonconstant rational function $f(t)=(a_1-b_1t+\tfrac12 c_1 t^2)^{-1} (a_2-b_2t+\tfrac12 c_2 t^2)$ satisfies a nontrivial polynomial equation $\sum_{j=0}^{k-1} \mu_j (f(t))^j=0$ for all $t \in \R$ (for which $f(t)$ is defined), a contradiction. \qedhere
\end{enumerate}
\end{proof}

It follows from  Lemma \ref{L:facts}\eqref{I:e1} that $\h \subset \g'$. Let $d=\dim (\h \cap \n)$. We have $d\leq\dim (\n)= 3$.
First of all, note that $d$ cannot equal $3$. Indeed, otherwise we would have $X_{2k+2},X_{2k+3}\in \h $,  contradicting Lemma \ref{L:facts}\eqref{I:eiei+1}.
Therefore $d \le 2$.

Let $N \in \n$ be a unit vector orthogonal to $\h \cap \n$. As $\h$ is a subalgebra and as $[\h,\h] \subset \n$, we have $\<[\h,\h], N\>=0$. As $\n$ is in the center of $\g'$, we obtain $\<J_NV, V\>=0$, where $V=\pi_\b(\h)$ and $\pi_\b \in \End(\g')$ is the orthogonal projection to $\b$. It follows that $\dim (\h) = \dim (\pi_\b(\h)) + \dim (\ker(\pi_\b) \cap \h)= \dim (V)+d \le k + 2$, from Lemma~\ref{l:Kabc}\eqref{it:Kb}, with the equality achieved only when $\dim (V)= k$ and $d=2$.

So if $\dim (\h) = k+2$, then $d=2$ and $\h=\Span(v_1+\lambda_1 N, \dots, v_k+\lambda_k N,$ $ N_2, N_3)$, where the vectors $v_i \in \b$ are linearly independent, $\Span_{i=1}^k(v_i)=V$, and $\{N,N_2,N_3\}$ is an orthonormal basis for $\n$. Then $\h^\perp=\Span(E_1, w_1+\mu_1 N, \dots, w_k+\mu_k N)$, where $E_1$ is the unit vector orthogonal to $\g'$, and $w_j \in \b$ for $j=1, \dots, k$. Denote $W=\Span_{i=1}^k(w_i)=\pi_\b(\h^\perp)$. Although the vectors $w_j$ may not be linearly independent, we have $V+W=\b$, so $\dim (W) \ge k-1$. We have $\dim (W) \le k$, so $\dim (W)$ is either $k$ or $k-1$.

Now by equation \eqref{L:basic}, with $X=w_j+\mu_jN, \; Y=v_i+\lambda_iN$ and $Z=N_\al$, $\al=2,3$, we obtain $\<[w_j,v_i],N_\al\>=0$, so $\<J_{N_\al}V,W\>=0$, for $\al=2,3$.

If $\dim (W)=k$, then $\dim (J_{N_\al}V) \le k-1$, so $\ker (J_{N_\al}) \cap V \not= 0$. By  Lemma~\ref{l:Kabc}\eqref{it:Ka}, $\dim (\ker (J_{N_\al}))=1$, and so $\ker (J_{N_\al}) \subset V$. By the same reasoning, $\ker (J_{N_\al}) \subset W$ and so $\ker (J_{N_\al}) \subset V \cap W$. Since $\dim(\b)=2k-1$ and $\dim (V)=\dim (W)=k$, we have $\dim (V\cap W)=1$. Hence  $\ker (J_{N_\al})=V \cap W$, for $\al=2, 3$. But this contradicts  Lemma~\ref{l:Kabc}\eqref{it:Kc}.

Next suppose that $\dim (W)=k-1$. Specifying the basis for $\h^\perp$ we can assume that $w_k=0$ and that $w_1, \dots, w_{k-1} \in \b$ are linearly independent, so $\h^\perp=\Span(E_1, w_1, \dots, w_{k-1}, N)$. Then $\<w_j, v_i\>=0$, so $W$ is the orthogonal complement to $V$ in $\b$. It now follows from $\<J_{N_\al}V,W\>=0$, for $\al=2,3$, that $V$ and $W$ are complementary invariant subspaces of both $J_{N_2}$ and $J_{N_3}$, hence of any $J_Z$, where $Z$ is a nonzero linear combination of $N_2$ and $N_3$. By Lemma~\ref{l:Kabc}\eqref{it:Ka}, $\dim(\ker (J_Z))=1$. As the projections of the kernel to invariant subspaces again lie in the kernel, we obtain that $\ker (J_Z)$ is a subspace of either $V$ or $W$. By continuity, the union $\mathcal{U}$ of the kernels of $J_Z$, taken over all nonzero $Z \in \Span(N_2,N_3)$, lies either in $V$ or in $W$. But by Lemma~\ref{l:Kabc}\eqref{it:Kd}, $\dim(\Span(\mathcal{U}))=k$, so $\Span(\mathcal{U})=V$ (as $\dim (V)=k$ and $\dim (W) = k-1$). Now take any two nonproportional $X, Y \in \mathcal{U}$. There exist nonzero vectors $Z_1=aN_2+bN_3, \; Z_2=cN_2+dN_3$ such that $\ker (J_{Z_1})=\Span(X), \; \ker (J_{Z_2})=\Span(Y)$ and moreover, by Lemma~\ref{l:Kabc}\eqref{it:Ka}, $Z_1$ and $Z_2$ are nonproportional, so $\Span(Z_1,Z_2)=\Span(N_2,N_3)$. As $\<J_{Z_1} X,Y\>=\<J_{Z_2} X,Y\>=0$, we obtain that $\<J_{N_2} X,Y\>=\<J_{N_3} X,Y\>=0$, for any pair of nonproportional vectors $X, Y \in \mathcal{U}$, hence, for arbitrary any pair of vectors $X, Y \in \mathcal{U}$, hence, for arbitrary vectors $X, Y \in V=\Span(\mathcal{U})$. It follows that $\<J_{N_\al}V,V\>=0$, for $\al=2,3$. As from the above, $\<J_{N_\al}V,W\>=0$, where $W$ is the orthogonal complement to $V$ in $\b$, we obtain that $\ker (J_{N_\al})$ contains the $k$-dimensional space $V,$ which strongly contradicts Lemma~\ref{l:Kabc}\eqref{it:Ka}.
\end{proof}


\section{Algebras from Table \ref{Table:2}}\label{S:table2}

 We treat the algebras $\g_{n,\alpha}$ in the order they appear in Table \ref{Table:2}.
 In this section there are numerous cases and subcases. For convenience, we adopt the following notational convention throughout this section: the expression $Y_i$ denotes an element of $\h$ of degree $i$. Similarly, the expression $Z_i$ denotes an element of $\hp$ of degree $i$. Furthermore, when we choose a basis $\{Y_{i_1},\dots,Y_{i_k}\ :\  i_1 < i_2 < \dots < i_k\}$, for $\h$, the assumption is that the elements are chosen so that $\langle Y_{i_j},E_{i_j}\rangle=1$ for each ${j}$, and that $Y_{i_j}$ has no component in the direction $E_{i_l}$ for $l>j$.  However we do not impose similar restrictions on the coefficients of our bases for $\hp$.

\begin{theorem}\label{T:g7}
Let $\h$ be a proper totally geodesic subalgebra of $\g_{7,\alpha}$. Then, for all $\alpha\in\R\backslash\{-2\}$, we have
$\dim(\h)\leq 3$.
\end{theorem}

\begin{proof}
By Remark \ref{R:codim}, our goal is to show that $\h$ cannot have dimension $4$. First suppose that $X_7\in\h$. As $\g_{7,\alpha}/\Span(X_7) \cong \m_2(6)\in\mathcal O_2$, we have $\dim(\h)\leq 3$ by Lemma~\ref{L:Xninh}\eqref{I:Xninhb}. Hence, we may suppose that $X_7\not\in\h$.

Note that if $\alpha\neq -1$, then $\g_{7,\alpha}\in\mathcal O_1$, and  Lemma \ref{L:Xnnotinh}\eqref{I:2.1a} gives the required result. So it remains to treat the case $\alpha= -1$. Suppose that $\h$ has dimension 4. Since $[X_3,X_4]=X_7\not\in\h$ and $E_1 \in \hp$ by Lemma~\ref{L:facts}\eqref{I:e1}, we can choose a basis so that $\h=\Span(Y_2,Y_i,Y_5,Y_6)$, where   $i=3$ or $4$.
Note that by Lemma \ref{L:facts}\eqref{I:ei+aei+1}, $Y_6=E_6$. Then we have $Y_5=E_5+a_5E_7$, for some $a_5\in \R$, where $a_5\neq 0$ by Lemma \ref{L:facts}\eqref{I:eiei+1}. So we may assume that $Y_2,Y_i$ have no component in the $E_5$ and $E_6$ directions. Moreover, the projection of $E_7$ to $\h$ is nonzero, as $a_5 \ne 0$ and has degree at least $5$, as it lies in the center of $\h$ by Lemma~\ref{L:g/Xn}\eqref{I:g/Xna} and as $[X_2, X_3]$ and $[X_2, X_4]$ are both nonzero. It follows that $\pi_\h (E_7) = \frac{a_5}{1+a_5^2} Y_5= \frac{a_5}{1+a_5^2} E_5+\frac{a_5^2}{1+a_5^2} E_7$, and so the vector $Z_5 = (1+a_5^2)(\pi_\h (E_7) - E_7)= a_5 E_5- E_7$ belongs to $\hp$. Then both $Y_2$ and $Y_i$ have no component in the $E_7$ direction.

Now if $i=3$, we have $Y_3=E_3$ by Lemma~\ref{L:facts}\eqref{I:ei+aei+1} and we may take $Y_2=E_2+c E_4$, where $c \ne 0$ by Lemma~\ref{L:facts}\eqref{I:eiei+1}. But then $Z_2 = E_4-cE_2 \in \hp$ and we get a contradiction with \eqref{L:basic}, as $[Z_2, E_6]=0$, but $\<[Z_2, Y_2], E_6\>=-(1+c^2)\<[E_2, E_4], E_6\> \ne 0$.

If $i=4$, then $Y_4=E_4$ and we may take $Y_2=E_2+bE_3$, so $Y_2=E_2$ by Lemma~\ref{L:facts}\eqref{I:ei+aei+1}. Then $E_3 \in \hp$ and we get a contradiction with \eqref{L:basic}, as $[Y_5, E_3]=0$, but $\<[E_3, E_4], Y_5\>=a_5\<[E_3, E_4], E_7\> \ne 0$.
\end{proof}


\begin{theorem}\label{T:g8}
Let $\h$ be a proper totally geodesic subalgebra of $\g_{8,\alpha}$. Then, for all $\alpha\in\R\backslash\{-2\}$, we have $\dim(\h)\leq 4$.
\end{theorem}

\begin{proof}
First suppose that $X_8\in\h$. Note that $\g_{8,\alpha}/{\Span(X_{8})}\cong \g_{7,\alpha}$. Consider  the  quotient map
$\pi:\g_{8,\alpha}\to \g_{7,\alpha}$.
By Lemma~\ref{L:g/Xn}\eqref{I:g/Xnb}, there is an inner product on $\g_{7,\alpha}$ for which $\overline\h:=\pi(\h)$  is a totally geodesic subalgebra. Thus by Theorem \ref{T:g7},  we have $\dim(\bar\h) \le 3$, so $\dim(\h)\leq 4$.
So we may suppose that $X_8\not\in\h$.
If $\alpha\neq 0$ we have that $\g_{8,\alpha}\in\mathcal O_1$ so Lemma \ref{L:Xnnotinh}\eqref{I:2.1a} implies that $\dim(\h)\leq 4$. It remains to treat the case $\alpha =0$. By Remark \ref{R:codim}, $\dim(\h)<6$. Assume that $\dim(\h)=5$. Since $[X_3,X_5]=X_8 \not\in\h$, it follows that $\h$ cannot have both an element of degree 3 and an element of degree 5. In fact,  we can write $\h=\Span(Y_2,Y_4,Y_5,Y_6,Y_7)$; indeed, we could not have an element $Y_3$ of degree 3 in $\h$ since otherwise $\deg([Y_2,Y_3])=5$. It follows that $\hp$ has an element of degree 2 or 3, by Lemma~\ref{L:facts}\eqref{I:eiei+1}. Moreover, $Y_7=E_7$, by Lemma \ref{L:facts}\eqref{I:ei+aei+1}. However, if there was $Z\in\hp$ such that $\deg(Z)=2$ then $\<[Z,Y_5],Y_7\>\neq 0$ while for $Z\in\hp$ of degree 3 we would have $\<[Z,Y_4],Y_7\>\neq 0$. In each case we get a contradiction with \eqref{L:basic}.
\end{proof}


\begin{theorem}\label{T:g9}
Let $\h$ be a proper totally geodesic subalgebra of $\g_{9,\alpha}$. Then, for all  $\alpha\neq -\frac 52, -2$, we have
 $\dim(\h)\leq 4$.
\end{theorem}

\begin{proof}
Note that if $\alpha\not\in\{-1, \frac 12\}$, then $\g_{9,\alpha}\in\mathcal O_1$ and Lemma \ref{L:Xnnotinh}(a) implies that $\dim(\h)\leq 4$ if $X_9\not\in\h$. Furthermore, $\g_{9,\alpha}/\Span(X_9)\cong\g_{8,\alpha}$ and if $\alpha\not\in\{-1,0\}$, then $\g_{8,\alpha}\in\mathcal O_2$ and Lemma \ref{L:Xninh}\eqref{I:Xninhb} implies that $\dim(\h)\leq 4$ if $X_9\in\h$. So there are three remaining cases:
\begin{enumerate}[\rm (a)]
\item $\alpha= 0$ and $X_9\in\h$,
\item $\alpha=-1$,
\item $\alpha= \frac 12$ and $X_9\not\in\h$.
\end{enumerate}

(a) Suppose $\alpha= 0$ and $X_9\in\h$. By Lemma \ref{L:g/Xn}\eqref{I:g/Xnb} and Theorem \ref{T:g8}, we have $\dim(\h) \leq 5$. Note that by Lemma \ref{L:facts}\eqref{I:eiei+1}, $\h$ has no elements of degree 8. First assume that there is a degree 2 element $Z$ in $\hp$. Then by Lemma \ref{L:facts}\eqref{I:kn-k}, there are no elements of degree 7 in $\h$.
From $\deg([Y_2,Y_3])=5$, $\deg([Y_2,Y_5])=7$, $\deg([Y_3,Y_5])=8$ for arbitrary elements $Y_i$ of degree $i=2,3,5$, it follows that there exists at most one element $i\in\{2,3,5\}$ with  $Y_i$ in $\h$.
Hence $\dim(\h)\leq 4$. Thus we may assume that  $E_2\in\h$. By Lemma \ref{L:facts}\eqref{I:eiei+1}, there is a degree 3 element $Z$ in $\hp$, while by Lemma \ref{L:facts}\eqref{I:kn-k} there are no elements of degree 6 in $\h$.
Since $\deg([E_2,Y_4])=6$, $\deg([Y_3,Y_5])=8$, $\deg([E_2,Y_3])=5$ for arbitrary elements $Y_i$ of degree $i=3,4,5$, it follows that $\h$ has no elements of degree $3$ and $4$. 
Therefore, $\dim(\h)\leq 4$.

(b) For $\alpha= -1$, suppose first that $X_9\in\h$. Then by Lemma \ref{L:facts}\eqref{I:eiei+1}, $\h$ has no elements of degree 8. If $Y_2 \in \h$, then $\dim (\h )\le 4$, as all the elements $[Y_2,Y_6], \, \ad^2(E_2)(Y_4)$ and $\ad^2(Y_3)(E_2)$ have degree $8$, so $\h$ has no elements of degree $3,4$ and $6$. Otherwise, $E_2 \in \hp$, so by Lemma \ref{L:facts}\eqref{I:kn-k}, we have no elements of degree 7 in $\h$. As $\deg([Y_3,Y_5])=8$, $\h$ cannot contain elements of both degree $3$ and degree $5$, so $\dim(\h)\leq 4$. Now suppose that $X_9\notin\h$. Since $[X_2,X_7]=-X_9$ and $[X_4,X_5]=-X_9$ we have  $\dim(\h)\leq 5$. If $\dim(\h)=5$, then $\h$ has elements $Y_3,Y_6,Y_8$ of degree 3,6,8 respectively, and in particular, by Lemma~\ref{L:facts}\eqref{I:ei+aei+1} we may take $Y_8= E_8\in\h$. From Lemma \ref{L:facts}\eqref{I:kn-k-1} we conclude that there are no elements of degree 2 in $\hp$, and so $E_2\in\h$. Hence by Lemma~\ref{L:facts}\eqref{I:eiei+1}, there is an element $Z\in\hp$ of degree 3. However, as $\deg([E_2,Y_3])=5$, we obtain a contradiction with Lemma~\ref{L:facts}\eqref{I:kn-k-1}.

(c) For $\alpha= \frac 12$ and $X_9\not\in\h$, the subalgebra $\h$  has  dimension $\leq 5$ as $[X_3,X_6]=[X_4,X_5]=\frac 12 X_9$. Suppose  $\dim(\h)=5$. Arguing as in case (b), $\h$ is spanned by $E_8$ and some $Y_2,Y_i,Y_j,Y_7$, where $i$ is 3 or 6 and $j$ is 4 or 5. First assume that $\hp$ has an element $Z$ of  degree 2. Then by Lemma~\ref{L:facts}\eqref{I:kn-k-1},  $\h$ has no elements of degree 6. Thus $i=3$ and since $\deg([X_2,X_4])=6$, we also have $j=5$. Then $\hp$ can have no elements of degree $3, 5$ and $6$ by Lemma~\ref{L:facts}\eqref{I:kn-k-1}. It follows that $\hp$ has an element $Z_4$ of degree 4. But by Lemma \ref{L:facts}\eqref{I:eiei+1}, $\<Y_7,E_9\>\neq 0$ and so $\<\nabla_{Y_5}Y_7,Z_4\>\neq 0$ contradicting  \eqref{L:basic}. So we may assume there is no elements of degree 2 in $\hp$. In this case, $E_2\in\h$ and by Lemma \ref{L:facts}\eqref{I:eiei+1}, there must be an element $Z_3\in\hp$ of degree 3. Then by Lemma \ref{L:facts}\eqref{I:kn-k-1}, $\h$ has no elements of degree 5. Consequently, since $[X_2,X_3]=X_5$, the algebra  $\h$ also has no elements of degree $3$. So $\h$ must have an element of degree $6$. However, $\<Y_7,E_9\>\neq 0$ by Lemma~\ref{L:facts}\eqref{I:eiei+1}, so $\<\nabla_{Y_6}Y_7,Z_3\>\neq 0$ contradicting \eqref{L:basic}. This completes the proof of the theorem.
\end{proof}


\begin{theorem}\label{T:g10}
Let $\h$ be a proper totally geodesic subalgebra of   $\g_{10,\alpha}$. Then, for all $\alpha\neq -\frac 52$, we have $\dim(\h)\leq 5$.
\end{theorem}

\begin{proof}
If $X_{10}\in\h$ and $\alpha\neq -2$ we may consider the quotient Lie algebra $\g_{10,\alpha}/\Span(X_{10})\cong\g_{9,\alpha}$ and the required result follows from that by Theorem \ref{T:g9}.

Suppose that $X_{10}\in\h$ and $\alpha=-2$. By Lemma \ref{L:facts}(b), there are no elements of degree 9 in $\h$. Note that $[X_3,X_6]=-2X_9$ and $[X_4,X_5]=3X_9$. If there is an element of degree 2 in $\hp$, Lemma \ref{L:facts}\eqref{I:kn-k} implies that there are no element in $\h$ of degree 8, from which it follows that $\dim(\h)\leq 5$. On the other hand, if there are no elements of degree 2 in $\hp$, there exists an element of degree 3 in $\hp$, by Lemma \ref{L:facts}(b). It follows that there are no elements of degree 7 in $\h$, again giving that $\dim(\h)\leq 5$.
So we may suppose that $X_{10}\not\in\h$.

If  $\alpha\not\in\{ \frac 12, -1\}$, then $\g_{10,\alpha}\in\mathcal O_1$, hence by Lemma \ref{L:Xnnotinh}\eqref{I:2.1a} we get $\dim(\h)\leq 5$. So there are two special cases that remain to be considered: $\alpha=\frac 12$ and $\alpha=-1$.

By Lemma~\ref{L:g/Xn}\eqref{I:g/Xna}, the vector $Y=\pi_\h (X_{10})$ lies in the center of $\h$. If $Y = 0$, then $X_{10} \in \hp$, so by Lemma~\ref{L:g/Xn}\eqref{I:g/Xnb} and Theorem~\ref{T:g9}, $\dim (\h) \le 4$, as $\g_{10,\alpha}/\Span(X_{10})\cong\g_{9,\alpha}$. Let $m=\deg (Y)$. Suppose that $m = 2$. As $Y$ lies in the center of $\h$ and as $[X_2,X_3], [X_2,X_4]$ and $[X_2,X_6]$ are all nonzero, $\h$ contains no elements of degree $3, 4$ or $6$, so $\dim (\h) \le 5$. Similarly, if $m = 3$, then $\h$ contains no elements of degree $2, 4$ or $5$, and if $m = 4$, then $\h$ contains no elements of degree $2, 3, 5$ or $6$. In both cases, $\dim (\h) \le 5$. Furthermore, if $m=5$ and $\alpha=\frac12$, then $\h$ contains no elements of degree $2, 3$ or $4$, so $\dim (\h) \le 5$, and if $m=5$ and $\alpha=-1$, then $\h$ contains no elements of degree $3$ or $4$, so $\dim (\h) \le 5$ unless $\h=\Span(Y_2,Y_5,Y_6,Y_7,Y_8,Y_9)$. But then we may take $Y_9=E_9$ by Lemma~\ref{L:facts}\eqref{I:eiei+1}, so $\hp$ has no elements of degree $2$ or $4$ by Lemma~\ref{L:facts}\eqref{I:kn-k-1}, and hence $\hp$ has at least two linearly independent elements in $\g_5$, contradicting the fact that $\dim(\h \cap \g_5)=5$ and $\dim(\g_5)=6$. Likewise, if $m=6$ and $\alpha=\frac12$, then $\h$ contains no elements of degree $2, 3$ or $4$, so $\dim (\h) \le 5$, and if $m=6$ and $\alpha=-1$, then $\h$ contains no elements of degree $2$ or $4$, so $\dim (\h) \le 5$ unless $\h=\Span(Y_3,Y_5,Y_6,Y_7,Y_8,Y_9)$. Then $E_2 \in \hp$ and we may take $Y_9=E_9$ by Lemma~\ref{L:facts}\eqref{I:eiei+1}. By Lemma~\ref{L:facts}\eqref{I:kn-k-1} we get a contradiction with the fact that $Y_7 \in \h$. Finally, if $m=7$ and $\alpha=-1$, then $\h$ contains no elements of degree $2$ or $3$, so $\dim (\h) \le 5$ unless $\h=\Span(Y_4,Y_5,Y_6,Y_7,Y_8,Y_9)$. But then $E_2, E_3 \in \hp$ and we may take $Y_9=E_9$ by Lemma~\ref{L:facts}\eqref{I:eiei+1}, which leads to a contradiction with Lemma~\ref{L:facts}\eqref{I:kn-k-1}, as $Y_7 \in \h$.

It remains to consider the cases when either $m=7$ and $\alpha=\frac12$ or $m=8$. Note that in both cases $\h$ contains the vector $Y = \pi_\h (E_{10})$ of degree $m$, and then $\hp$ contains the vector $Z=Y-E_{10}$, also of degree $m$. Consider two cases.


\emph{The Case $\alpha=-1$}. Then both $\h$ and $\hp$ contain an element of degree $8$, say $Y_8$ and $Z_8$ respectively. Since $[X_4,X_6]=-[X_3,X_7]=X_{10}\not\in\h$, the subalgebra $\h$ cannot have elements of both degree 3 and degree 7, and nor can it have elements of both degree 4 and degree 6. So $\dim(\h)\leq 6$.  Suppose that $\dim(\h)= 6$. Since $[X_2,X_4]= X_6$, we have $\h=\Span(Y_2,Y_3,Y_5,Y_6,Y_8,E_9)$ or $\h=\Span(Y_2,Y_5,Y_6,Y_7,Y_8,E_9)$. In both cases, the three-dimensional ideal $(\g_{10,-1})_8=\Span(X_8,X_9,X_{10}) \subset \g_{10,-1}$ contains linearly independent vectors $Y_8, E_9 \in \h$ and $Z_8 \in \hp$, hence $(\g_{10,-1})_8=\Span(Y_8,E_9,Z_8)$. Then by Lemma~\ref{L:g/Xn}\eqref{I:g/Xnb}, there is an inner product on the algebra $\g_{10,-1}/(\g_{10,-1})_8\cong\g_{7,-1}$, for which $\overline\h=\pi(\h)$  is a totally geodesic subalgebra. But $\dim (\overline\h)=4$, which contradicts Theorem~\ref{T:g7}.


\emph{The Case $\alpha=\frac 12$}. Then both $\h$ and $\hp$ contain either an element of degree~$7$ or an element of degree $8$.

As $[X_4,X_6]=\frac 12X_{10}\not\in\h$ and $E_1\in\hp$, we have $\dim(\h)\leq 7$.
Suppose first that $\dim(\h)=7$. Since $\h$ does not have both an element of degree 4 and an element of degree 6,
and since $[X_2,X_4]=\frac52 X_6$, we may write $\h=\Span(Y_2,Y_3,Y_5,\ldots, Y_9)$.
But then by Lemma~\ref{L:zn-2,zn-1}, $\hp$ has no elements of degree $\geq 7$, a contradiction.

Now suppose that $\dim(\h)=6$. We consider three subcases:
\begin{enumerate}
\item[(i)] there are no elements of degree 3 in $\h$,
\item[(ii)] there is an element $Y_3$ of degree 3 but no elements of degree 2 in $\h$,
\item[(iii)] there are elements $Y_2,Y_3\in\h$ of degree 2 and 3 respectively.
\end{enumerate}

{\it Subcase} (i): Arguing as above, since $\dim(\h)=6$, $[X_2,X_4]=\frac 52 X_6$ and $[X_4,X_6]=\frac 12 X_{10}\not\in\h$, we may write $\h=\Span(Y_2,Y_5,Y_6,Y_7,Y_8,E_9)$, so by Lemma~\ref{L:zn-2,zn-1}, $\hp$ has no elements of degree $7$ or $8$, a contradiction.

{\it Subcase} (ii): We have $E_2\in\hp$. As $\dim(\h)=6$ we have $\h=\Span(Y_3,Y_4,Y_5,$ $Y_7,Y_8,E_9)$ or $\h=\Span(Y_3,Y_5,Y_6,Y_7,Y_8,E_9)$. In the latter case, $\hp$ has no elements of degree $7$ or $8$, by Lemma \ref{L:zn-2,zn-1}, a contradiction. In the former case, suppose for the moment that there exists $Z_8\in\hp$. As $\Span(Y_8,E_9,Z_8)=\Span(E_8,E_9,E_{10})$, we can choose $Y_7 = E_7$ and then $Y_5 \in \Span(E_5,E_6)$, so $E_5\in\h$, by Lemma~\ref{L:facts}\eqref{I:ei+aei+1}. But since $[X_2,X_7]=[X_2,X_8]=[X_3,X_7]=0$, we have $[E_2,E_7]=0$ and hence $2\<\nabla_{E_5}E_7,E_2\>=\<[E_2,E_5],E_7\>\neq 0$, contradicting \eqref{L:basic}. It follows that $\hp$ has no elements of degree 8, and therefore, by Lemma \ref{L:facts}\eqref{I:kn-k-1}, we have $\hp=\Span(E_1,E_2,Z_3,Z_7)$. Arguing as before, we obtain $E_5\in\h$.  We have $Y_4:=E_4+a_{4,6}E_6+a_{4,10}E_{10}$ for some $a_{4,6}, a_{4,10}\in\R$. Then  $\<Y_4,Z_7\>=0$ implies $a_{4,10}=0$. Then by Lemma~\ref{L:facts}\eqref{I:eiei+1}, $a_{4,6}\neq 0$ as $E_5\in\h$. However, $\<\nabla_{Y_4}Y_4,E_2\>=a_{4,6}\<[E_2,E_4],E_6\>\neq 0$,
contradicting \eqref{L:basic}.

{\it Subcase} (iii): Here $\h$ contains an element $[Y_2,Y_3]$ of degree $5$ and elements $\ad^2(Y_2)(Y_3)$ and $\ad^2(Y_3)(Y_2)$ of degree $7$ and $8$ respectively. Moreover, $\h$ has no elements of degree $4$, by the same argument used in subcase (i), and no element $Y_6$ of degree $6$, as otherwise $\h$ would also contain an element $[Y_3,Y_6]$ of degree $9$ contradicting the fact that $\dim (\h) = 6$. So $\h=\Span(Y_2,Y_3,Y_5,Y_7,Y_8,E_9)$. We may assume that $Y_5:=E_5+a_{5,6}E_6+a_{5,10}E_{10}$ and $Y_i:=E_i+a_{i,10}E_{10}$, $i=7,8$. By Lemma~\ref{L:facts}\eqref{I:kn-k-1} there are no elements of degree 4 or 6 in $\hp$. Moreover, by the same argument used in subcase (ii), $\hp$ has no elements of degree 8. It follows that $\hp$ contains an element $Z_7$ of degree $7$. Then we have $E_5\in\h$ and so $\hp=\Span(E_1,Z_2,Z_3,Z_7)$. Moreover, without loss of generality we may assume that $\<Y_3,E_j\>=0$ for  $j= 7,8,9,10$. The contradiction with \eqref{L:basic} is then obtained from
\begin{equation*}
0=2\<\nabla_{Y_3}E_5,Z_2\>=\<[Z_2,Y_3],E_5\>+\<[Z_2,E_5],Y_3\>=\<[Z_2,Y_3],E_5\>.\qedhere
\end{equation*}
\end{proof}


\begin{theorem}\label{T:g11}
Let $\h$ be a proper totally geodesic subalgebra of  $\g_{11,\alpha}$. Then, for all $\alpha\not\in\{-\frac 52, -1,-3\}$, we have
$\dim(\h)\leq 5$.
\end{theorem}

\begin{proof}
Consider two cases: $X_{11}\in\h$ and $X_{11}\not\in\h$.

\emph{Case $X_{11}\in\h$}:
Consider the quotient map $\pi:\g_{11,\alpha}\to \g_{11,\alpha}/\Span(X_{11})$. For $\alpha\not\in\{ -2,0,\frac 12\}$, we have $\g_{11,\alpha}/\Span(X_{11})\cong \g_{10,\alpha}\in \mathcal O_2$,
and Lemma \ref{L:Xninh}\eqref{I:Xninhb} gives $\dim(\h)\leq 5$.
Let $\alpha\in\{ -2,0,\frac 12\}$. Lemma \ref{L:g/Xn}\eqref{I:g/Xnb} and Theorem \ref{T:g10} give $\dim(\pi(\h))\leq 5$ and so
$\dim(\h)\leq 6$.  Assume $\dim(\h)= 6$. Note that $\h$ has no elements of degree $10$ by Lemma \ref{L:facts}\eqref{I:eiei+1}. 

If $\alpha\not=\frac12$, then $[X_2,X_8],[X_3,X_7],[X_4,X_6]$ are all nonzero multiplies of $X_{10}$, so $\h$ cannot have elements of both degree $2$ and degree $8$, nor can it have elements of both degree $3$ and degree $7$, nor of both degree $4$ and degree $6$. Then $\h$ necessarily contains elements $Y_5,Y_9$. By Lemma~\ref{L:facts}\eqref{I:ei+aei+1} we may take $Y_9=E_9$. Then by Lemma \ref{L:facts}\eqref{I:kn-k}, $\hp$ has no elements of degree $2$, and so we have $E_2\in\h$. Consequently, $\h$ has no elements of degree $8$. Since $Y_5\in\h$ and $[X_3,X_5]=X_8, \; [X_2,X_5]=(1+\alpha) X_7$, we conclude that $\h$ has an element of degree $7$, but no elements of degree $3$. It follows that $\h=\Span(E_2,Y_j,Y_5,Y_7,E_9,E_{11})$, where $j=4$ or $6$.

First suppose $\alpha=-2$. Since $[X_2,X_6]=-2X_8$ and $\h$ has no elements of degree 8, $\h$ has no elements of degree 6. Hence $\h$ has an element of degree 4. Then by Lemma \ref{L:facts}\eqref{I:kn-k}, $\hp$ has no elements of degree 2, 4, 6, 7 or 9. So we can write $\hp=\Span(E_1,E_3,Z_5,Z_8,E_{10})$. But then since $[X_4, X_5] =3X_9$, we have $\<\nabla_{Y_4}E_9,Z_5\>\neq 0$, contradicting \eqref{L:basic}.

Now suppose $\alpha=0$. Since $\h$ does not have elements of both degree 4 and 6, and since $[X_2,X_4]=2X_6$, the subalgebra $\h$ has no elements of degree 4, and hence it must have one of degree 6.  So we have $\h=\Span(E_2,Y_5,Y_6,Y_7,E_9,E_{11})$. But since $E_3\in\hp$ and $[X_3, X_6] $ is a nonzero multiple of $X_9$, we have $\<\nabla_{Y_6}E_9,E_3\>\neq 0$, again contradicting \eqref{L:basic}.

Now suppose $\alpha=\frac 12$. Note that $\h$ does not have elements of both degree 4 and 6. First assume that $\hp$ has no elements of degree 2. So $E_2\in\h$. Note that $\hp$ has an element $Z_3$, as otherwise we would have $E_2,E_3\in\h$, contradicting Lemma \ref{L:facts}\eqref{I:eiei+1}. But then as $[X_3,X_8]$ is a nonzero multiple of $X_{11}$, $\h$ has no elements of degree $8$ by Lemma~\ref{L:facts}\eqref{I:kn-k}. Since $[X_2,X_4]=\frac52 X_6, \; [X_2,X_6]=\frac12 X_8$ and $[X_2,X_3]=\frac52 X_5, \; [X_3,X_5]=X_8$, we conclude that $\h$ has no elements of degree $3, 4, 6$ or $8$, giving $\dim (\h) \le 5$. So we may suppose that $\hp$ has an element $Z_2$. It then follows from Lemma~\ref{L:facts}\eqref{I:kn-k} that $\h$ has no elements of degree $9$. Consequently $\h$ does not have elements of both degree 3 and 6, and nor does it have elements of both degree 4 and 5. Then $\h$ necessarily has elements $Y_2, Y_7, Y_8$ and has no elements of degree $4$, as $\ad^2(X_4)(X_2)=-\frac54 X_{10}$. We conclude that either $\h=\Span(Y_2,Y_3,Y_5,Y_7,Y_8,E_{11})$ or $\h=\Span(Y_2,Y_5,Y_6,Y_7,Y_8,E_{11})$. Then by  Lemma \ref{L:facts}\eqref{I:kn-k}, we have $\hp=\Span(E_1,Z_2,Z_5,Z_7,E_{10})$ or $\hp=\Span(E_1,Z_2,Z_7,Z_8,E_{10})$ respectively. For an ideal $(\g_{11,\frac 12})_{10}=\Span(E_{10},E_{11})$, we have $\g_{11,\frac 12}/(\g_{11,\frac 12})_{10}\cong \g_{9,\frac 12}$. But then by Lemma \ref{L:g/Xn}\eqref{I:g/Xnb}, the image of $\h$ in $\g_{9,\frac 12}$ is a totally geodesic subalgebra of dimension 5, contradicting Theorem \ref{T:g9}. This completes the proof in the case $X_{11}\in\h$.

\emph{Case $X_{11}\not\in\h$}:
Let $\alpha_1 \approx -1.5919$ and $\alpha_2 \approx 1.5342$ be the (unique) real roots of the polynomials $2\alpha^3+2\alpha^2+3$ and $4\alpha^3+8\alpha^2-8\alpha-21$ respectively. By Lemma \ref{L:Xnnotinh}\eqref{I:2.1a}, the required result holds provided $\alpha\neq -\frac 14, \, \alpha_1, \, \alpha_2$. In these special cases, $\dim(\h)\leq 6$, as $\h$ cannot contain elements of both degree $i$ and $11-i$ when $[X_i,X_{11-i}]$ is a nonzero multiple of $X_{11}$. So, suppose that $\dim(\h)=6$. Then $\h$ must necessarily have an element $Y_{10}$ and we may take $Y_{10}=E_{10}$ by Lemma~\ref{L:facts}\eqref{I:ei+aei+1}.

First suppose $\alpha =-\frac 14$. Then $\h=\Span(Y_5,Y_6,E_{10},Y_i,Y_j,Y_k)$, where $i \in \{2,9\}, \, j \in \{3,8\}$, $k \in \{4,7\}$. If there is an element $Z_2\in\hp$, then by Lemma \ref{L:facts}\eqref{I:kn-k-1}, there are no elements of degree 8 in $\h$, and so we have an element $Y_3\in\h$. But then $\h$ contains an element $[Y_3,Y_5]$ of degree 8, a contradiction. On the other hand, if $E_2\in\h$, then $\h$ contains an element $[E_2,Y_5]$ of degree $7$. But then $\hp$ has no elements of degree $3$ by Lemma \ref{L:facts}\eqref{I:kn-k-1}, so $E_2,E_3\in\h$, which contradicts Lemma~\ref{L:facts}\eqref{I:eiei+1}.

Now suppose $\alpha=\alpha_1$, so that $2\alpha^3+2\alpha^2+3= 0$. Then there are elements $Y_2,Y_9\in\h$. If there is an element $Z_2\in\hp$, then there are no elements of degree 8 in $\h$ by Lemma~\ref{L:facts}\eqref{I:kn-k-1}, so there is necessarily an $Y_3\in\h$. But then $\h$ contains an element $\ad^2(Y_3)(E_2)$ of degree $8$, a contradiction. On the other hand, if $\hp$ has no elements of degree 2, then $E_2\in\h$, so by Lemma~\ref{L:facts}\eqref{I:eiei+1}, there is an element $Z_3\in\hp$. Then $\h$ has no elements of degree $7$ by Lemma~\ref{L:facts}\eqref{I:kn-k-1}, hence no elements of degree $3$ or $5$, as $\ad^2(X_2) (X_3)$ and $\ad(X_2) (X_5)$ are nonzero multiples of $X_7$. Then $\h=\Span(E_2,Y_4,Y_6,Y_8,Y_9,E_{10})$. We may take $Y_9=E_9+aE_{11}$. Then $a\neq 0$ by Lemma~\ref{L:facts}\eqref{I:eiei+1}, and in particular, $X_{11} \notin \hp$. It follows that $\hp=\Span(E_1,Z_3,Z_5,Z_7,Z_9)$. But as $[X_5, X_6]$ is a nonzero multiple of $X_{11}$, we have $2\<\nabla_{Y_6}(E_9+aE_{11}),Z_5\>=a\<[Z_5,Y_6],E_{11}\>\neq 0$, contradicting \eqref{L:basic}.

Finally, suppose $\alpha=\alpha_2$, so that $4\alpha^3+8\alpha^2-8\alpha-21= 0$. Then there exist $Y_3,Y_8\in\h$. By Lemma \ref{L:facts}\eqref{I:kn-k-1}, $\hp$ has no elements of degree $2$, so $E_2\in\h$. But then $\h$ contains an element $\ad^4(E_2)(Y_3)$ of degree $11$, which is a contradiction.
\end{proof}

\bibliographystyle{amsplain}

\end{document}